\documentclass[12pt, reqno]{amsart}
\usepackage{mathrsfs}
\usepackage{amsfonts}
\usepackage{}
\usepackage{amsmath}
\usepackage{amsthm}
\usepackage{amssymb}
\usepackage[compress,numbers]{natbib}
\usepackage{graphicx}
\usepackage{comment}
\usepackage{diagbox}
\usepackage{tikz}
\usepackage{booktabs}
\usepackage{makecell}
\usepackage{diagbox}
\usepackage{graphicx}
\usepackage{float}
\usepackage{multirow}

\usetikzlibrary{shapes, positioning}
\newtheorem{theorem}{Theorem}[section]
\newtheorem{lemma}[theorem]{Lemma}

\theoremstyle{definition}

\newtheorem{example}[theorem]{Example}

\newtheorem{proposition}[theorem]{Proposition}
\newtheorem{corollary}[theorem]{Corollary}

\theoremstyle{remark}
\newtheorem{remark}[theorem]{Remark}

\numberwithin{equation}{section}

\newcommand{\R}{\mathbb{R}}
\newcommand{\conv}{\operatorname{conv}}

\newcommand{\Rn}{\mathbb{R}^{n}}
\newcommand{\Sn}{S^{n-1}}
\newcommand{\On}{\mathrm{O}(n)}
\newcommand{\SOn}{\mathrm{SO}(n)}

\newcommand{\Lp}{L_{p}}

\newcommand{\la}{\langle}
\newcommand{\ra}{\rangle}
\newcommand{\clG}{\overline{G}}

\newcommand{\KG}{\mathcal{K}_{G}}
	
\begin{document}
	
	\title{The  $L_{p}$ Dual  Minkowski  Problem for Group-Invariant Convex Bodies  }
	
	\author{Junjie Shan }
	
	\address{School of Mathematics, Sichuan University, Chengdu, Sichuan, 610064, P. R. China}

	\email{shanjjmath@163.com 
	}

	\thanks{{\it 2010 Math Subject Classifications}:  52A40, 52A38}
	

	\begin{abstract}
	In this paper, we study the $L_p$ dual Minkowski problem for all $q, p \in \mathbb{R}$ from an algebraic perspective. We establish the existence of solutions for group-invariant convex bodies (not necessarily origin-symmetric), thereby covering three fundamental problems as special cases: the $L_p$ Minkowski problem ($q = n$), the $L_p$ Aleksandrov problem ($q = 0$), and the dual Minkowski problem ($p = 0$).
	\end{abstract}
	
	\maketitle
	   
	  \section{Introduction} 
	  The central problem in convex geometry involves characterizing geometric measures  of convex bodies. The classical Minkowski problem seeks necessary and sufficient conditions for a given Borel measure on the unit sphere to be the surface area measure of a convex body. This is an influential problem not only in geometric analysis but also in fully nonlinear partial differential equations.
	  

	  Since the classical Minkowski problem, many extensions have been introduced and extensively studied. The $L_p$ Minkowski problem, which characterizes the $L_p$ surface area measure, was first introduced by Lutwak \cite{Lutwak1993,Lutwak1996}. When $p = 1$, the $L_p$ Minkowski problem reduces to the classical Minkowski problem.  The $L_p$ Minkowski problem has been solved for $p\ge1$ (see  \cite{lyz2004,Lutwak1993,L-O 1995}). However, for $p<1$, the family of $L_p$ Minkowski problems contains many important unsolved cases, notably the logarithmic Minkowski problem ($p=0$) and the centro-affine Minkowski problem ($p=-n$), which remain significant open problems, though partial results appear in \cite{BHZ2016,CLZ2019,JJZ2016,LW2013,zhu2014,zhu2015}.  Particularly for $p<0$, little is known, with progress mainly in discrete cases (\cite{zhu2015,zhu2017}) or absolutely continuous cases: results for $-n<p<0$ with densities in $L_{\frac{n}{n+p}}(S^{n-1})$ \cite{B-B-A-Y 2019,chou-wang}, and results under regularity assumptions for $p<-n$ \cite{GLW}. Some affine isoperimetric inequalities were obtained by the solution of the $\Lp$ Minkowski problem, see \cite{CGlpine,HFGlpine,LYZ2000lp,lyz2002}.
	  
	  In the groundbreaking work \cite{HLYZdual}, Huang-LYZ introduced a new family of geometric measures called dual curvature measures, which can be viewed as differentials of dual quermassintegrals. The dual Minkowski problem, which characterizes the dual curvature measure, remains widely open. While this problem is completely solved for even data when $0 \leq q \leq n$,  little is known about the case when $q > n$.  When 
	  $q=0$, the $0$-th dual curvature measure corresponds to the well-known integral curvature measure. Its $L_p$ generalization, namely the $L_p$ integral curvature measure, and the corresponding $L_p$ Aleksandrov problem, were posed in \cite{HLYZ2018}.
	   For progress  on this question, see \cite{BHP2018,BLYZZ2019,chenli 2018,EH2023,HLXZ,lsw,Lyzdualcone,zhao2017,zhao2018}.

	  The $(p, q)$-th dual curvature measure, introduced in \cite{lyz2018}, unifies the aforementioned $L_p$ surface area measure, the $L_p$ integral curvature measure, and the dual curvature measure, thereby establishing a novel connection between these previously disparate concepts. Consequently, the corresponding Minkowski problems for these geometric measures are also unified. The $\Lp$ dual Minkowski problem concerns prescribing the
	    $(p, q)$-th dual curvature measure $\widetilde{C}_{p,q} (K, Q, \cdot) $ of a convex body $K$ relative to a star body $Q$ in $\Rn$ (see Section 2 for a precise definition). This problem is formally stated as follows:
	    
	    \textit{Given a nonzero finite Borel measure $\mu$ on the unit sphere $S^{n-1}$, a star body $Q$, and real numbers $p,q$, what are necessary and sufficient conditions  for the existence of a convex body $K$ satisfying 
	    	$\mu = \widetilde{C}_{p,q}(K, Q, \cdot)$?}
	  
	  When the  measure $\mu$ has a density function $f$, the aforementioned $L_p$ dual Minkowski problem reduces to the following Monge-Amp\`{e}re type equation:
	  \begin{equation}\label{continuous}
	  	\det(\nabla_{ij} h + h \delta_{ij}) = h^{p-1} \rho_{Q}^{q-n}(\nabla h) f,
	  \end{equation}
	  where $\rho_{Q}$ is the radial function of $Q$, $h$ is the unknown support function, $\delta_{ij}$ denotes the Kronecker delta, $\nabla h$ is the Euclidean gradient and $\nabla_{ij}h$ is the Hessian of $h$ on $S^{n-1}$.

	  
	   The case $p > 0$, $q < 0$ of the $\Lp$ dual Minkowski problem,   similar to the classical Minkowski problem, was solved by Huang-Zhao \cite{huangzhao dual}. They also studied the case for even measures $\mu$ when $p, q > 0$ with $p \neq q$.   B\"or\"oczky-Fodor \cite{BF2019} extended these even solutions to the general case when $p > 1$, $q > 0$ and $p \neq q$.  Chen-Li \cite{chenli flow} proved the case for $p > 0$, $q \neq n$ using  Gauss curvature flow.  
	   For further references,  see \cite{ai lpdual,liwan2025,LLL2022,LP2021,mui2022}.

	  The case $ q>0, p\le0$  of the $\Lp$ dual Minkowski problem is more challenging and remains poorly understood. 
	  Even for the special cases mentioned earlier, the problem is widely open. For example, when $q=n$, the $(p, n)$-th dual curvature measure $\widetilde{C}_{p,n}(K, Q, \cdot)$ equals the $\Lp$ surface area measure (up to a factor $n$). In this setting, the $\Lp$ dual Minkowski problem reduces to the $\Lp$ Minkowski problem, which includes both the  logarithmic Minkowski problem (solved in the even case by B\"or\"oczky-LYZ \cite{KLYZ2013JAMS}) and the centro-affine Minkowski problem. 
	  When $p=0$, the $\Lp$ dual Minkowski problem coincides with the dual Minkowski problem, and remains largely open for $q > n$.

	  For the  challenging case  $q>0, p\le 0$  of the $\Lp$ dual Minkowski problem,   Mui \cite{mui2024} solved the range $-1 < p < 0$ with $q < 1 + p$ and $p \neq q$ for even data. For $q>n-1$ and some negative $p$, Guang-Li-Wang \cite{GLW2023 flow} established existence of solutions under regularity assumptions.  Very recently, from a group-symmetry perspective, B\"or\"oczky-Kov\'{a}cs-Mui-Zhang \cite{BKMZ group} solved the continuous $\Lp$ dual Minkowski problem \eqref{continuous} for $q>0$, $-q^{*}<p<0$, where the density $f$ is $G$-invariant  with $f\in L^{s}(\Sn) $ for specified $s$. Here, $G$ is a closed subgroup of $\On$ without non-zero fixed point, and $q^*$ denotes a $q$-dependent constant.

	  	This paper aims to solve the $\Lp$ dual Minkowski problem  from an algebraic perspective, combining convex geometry and group representation theory.  We establish the existence of solutions to the $L_p$ dual Minkowski problem for all $q ,p \in \mathbb{R}$  in the setting of
	  	 group-invariant measures and  group-invariant convex bodies.   Thus, the Minkowski problem serves as a bridge connecting algebra, geometry, and analysis.
	  	As noted, this parameter range encompasses several  major open problems.

We recall fundamental notions from convex geometry.	    Let $\mathcal{K}^n$ denote the set of convex bodies in $\mathbb{R}^n$, 
	    $\mathcal{K}_o^n$ the set of convex bodies in $\mathbb{R}^n$ containing the origin in their interior, 
	    and $\mathcal{K}_e^n$ the set of origin-symmetric convex bodies in $\mathbb{R}^n$. 
	    Let $\mathcal{B}^n$ denote the class of Euclidean balls centered at the origin in $\mathbb{R}^n$.

	    	Given a subgroup $G$ of the orthogonal group $\On$,
	    we denote by $\mathcal{K}_{G}$ the collection of all \textit{$G$-invariant convex bodies}, i.e.,
	    \begin{equation*}
	    	\mathcal{K}_{G} := \left\{ K \in \mathcal{K}^n : gK = K \text{ for all } g \in G \right\}.
	    \end{equation*}
	   A Borel measure $\mu$ on $S^{n-1}$ is \textit{$G$-invariant} if $
	   \mu(gE) = \mu(E)$ for all $g \in G$ and any Borel set  $E \subset S^{n-1}$.

	   Group-invariant convex bodies are the primary focus of this paper. We first characterize relationships between $\mathcal{K}_G$ and fundamental classes of convex bodies in convex geometry, establishing a complete classification in Section \ref{classification}. For example, 	let $ G $ be a  subgroup of  $ \On $.	Then, $\mathcal{K}_{G} \subset \mathcal{K}^{n}_{e}$ if and only if 	$-x \in \overline{Gx}$ for all $x\in \Rn$; $\mathcal{K}_{G} \subset \mathcal{K}^{n}_{o}$ if and only if $G$ has no non-zero fixed points  in $\Rn$. Moreover, we have:
	   
	   

	   \begin{table}[h]
	   	\centering
	   	\begin{tabular}{|c|c|c|c|c|}
	   		\hline
	   		 \diagbox[width=5.5em, height=1.8em]{}{$\mathcal{C}$} & $\mathcal{B}^n$ & $\mathcal{K}_e^n$ & $\mathcal{K}_o^n$ & $\mathcal{K}^n$ \\ 
	   		\hline
	   		$\mathcal{K}_{G} = \mathcal{C}  \iff$ & 
	   		\begin{minipage}{0.18\textwidth}
	   			\vspace{0.2em}
	   			$\overline{Gv}=\Sn$ for all  $v\in S^{n-1}$
	   			\vspace{0.2em}
	   		\end{minipage} & 
	   		\begin{minipage}{0.18\textwidth}
	   			$G = \{\pm I\}$
	   		\end{minipage} & 
	   		\begin{minipage}{0.18\textwidth}
	   			No such $G$
	   		\end{minipage} &
	   		\begin{minipage}{0.18\textwidth}
	   			$G=\{I\}$
	   		\end{minipage} \\ 
	   		\hline
	   		$\mathcal{K}_{G} \subset \mathcal{C} \iff$ & 
	   		\begin{minipage}{0.18\textwidth}
	   			\vspace{0.2em}
	   				$\overline{Gv}=\Sn$ for all  $v\in S^{n-1}$
	   				\vspace{0.2em}
	   		\end{minipage} & 
	   		\begin{minipage}{0.18\textwidth}
	   			\vspace{0.2em}
	   			$-x \in \overline{Gx}$ for all $x\in \Rn$
	   			\vspace{0.2em}
	   		\end{minipage} & 
	   		\begin{minipage}{0.18\textwidth}
	   			\vspace{0.2em}
	   			No non-zero fixed points
	   			\vspace{0.2em}
	   		\end{minipage} &
	   		\begin{minipage}{0.18\textwidth}
	   			Always true
	   		\end{minipage} \\ 
	   		\hline
	   		$\mathcal{K}_{G} \supset \mathcal{C} \iff$ & 
	   		\begin{minipage}{0.18\textwidth}
	   			\vspace{0.2em}
	   			Always true
	   			\vspace{0.2em}
	   		\end{minipage} & 
	   		\begin{minipage}{0.18\textwidth}
	   			\vspace{0.2em}
	   			$G = \{I\}$ or $\{\pm I\}$
	   			\vspace{0.2em}
	   		\end{minipage} & 
	   		\begin{minipage}{0.18\textwidth}
	   			$G = \{I\}$
	   		\end{minipage} &
	   		\begin{minipage}{0.18\textwidth}
	   			$G=\{I\}$
	   		\end{minipage} \\
	   		\hline
	   	\end{tabular}
	   \end{table} 
	  where $Gv = \{ gv : g \in G \}$ is the orbit and $\overline{Gv}$ its closure. $\mathcal{C}$ denotes the convex body class in each column header.


	

We study the $\Lp$ dual Minkowski problem for $G$-invariant convex bodies, where $G$ is an irreducible group.
	We say a group $G$ is \textit{irreducible} if its action on $\mathbb{R}^n$ is irreducible  (i.e., no nontrivial subspace $V \subset \mathbb{R}^n$ satisfies $gV = V$ for all $g \in G$). In other words, 	 $G $ is irreducible if the only subspaces of $ \mathbb{R}^n$ that are invariant under $G$  are $\{0\}$ and $\mathbb{R}^n$.

From an algebraic perspective, origin-symmetric convex bodies  can be viewed as $\{\pm I\}$-invariant convex bodies, where $I$ is the identity map. Note that $\{\pm I\}$ is not irreducible. For example, when $n=2$, any line passing through the origin is an invariant subspace under $\{\pm I\}$. 
In contrast, the rotation symmetry group $G$ of an equilateral triangle, consisting of rotations by $0^\circ$, $120^\circ$, and $240^\circ$, is clearly irreducible. The $G$-invariant convex bodies  are precisely those convex bodies that remain unchanged under the $120^\circ$  rotation.

	For a specific irreducible group $G \subset \On$, the set of $G$-invariant convex bodies and the set $\mathcal{K}_e^n$ may be such that neither contains the other (e.g., $G$ is the symmetry group of a regular simplex), or the former may be contained in the latter (e.g., $G$ is the symmetry group of a cube):
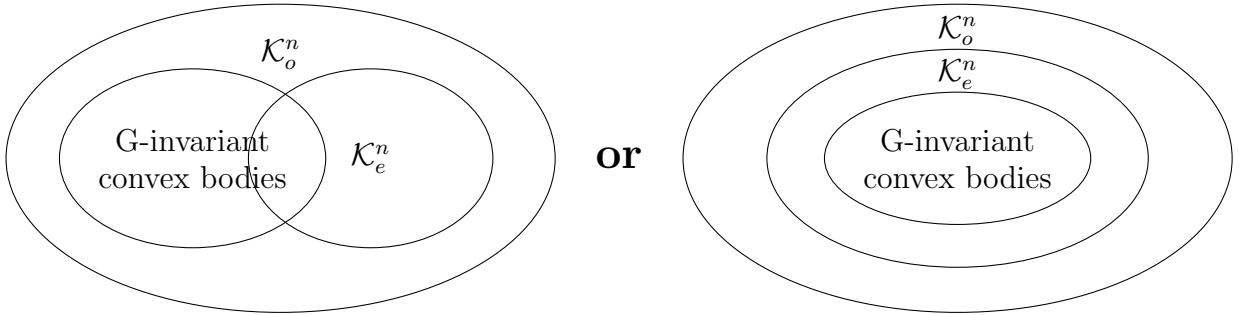
\begin{figure}[H]
	\scalebox{0.78}
	\centering
	\begin{tikzpicture}[
		set/.style={
			ellipse,
			draw,
			minimum width=3.25cm,
			minimum height=2.38cm,
			inner sep=0pt,
			align=center
		},
		label/.style={font=\small, align=center}
		]
		
		\begin{scope}[xshift=-4.5cm, scale=0.9]
			\node[set, minimum width=7.3cm, minimum height=4.1cm] (A) at (0,0) {};
			
			\node[set] (B) at (-1.3,0) {G-invariant \\ convex bodies};
			\node[set] (C) at (1.33,0) {$\mathcal{K}_e^n$};
			
			\node[above, yshift=-1.01cm] at (A.north) {$\mathcal{K}_o^n$};
		\end{scope}
		
		\node at (0,0) {\Large\textbf{or}};
		
		\begin{scope}[xshift=4.5cm, scale=0.9]
			\node[set, fill=white, minimum width=7.3cm, minimum height=4.1cm] (outer) at (0,0) {};
			\node[above, yshift=-0.7cm] at (outer.north) {$\mathcal{K}_o^n$};
			
			\node[set, fill=white, minimum width=5.07cm, minimum height=2.9cm] (middle) at (0,0) {};
			\node[above, yshift=-0.68cm] at (middle.north) {$\mathcal{K}_e^n$};
			
			\node[set, fill=white, minimum width=2.56cm, minimum height=1.76cm] (inner) at (0,0) {G-invariant \\ convex bodies};
		\end{scope}
	\end{tikzpicture}
\end{figure}
		

		
		The class of $G$-invariant convex bodies is  closed under the operations of   $L_p$ Minkowski addition $+_{p}$, scalar multiplication, polar  operation $*$, and convex hull operation  (if $K, L\in \mathcal{K}_{G}$, then $\conv(K, L) \in \mathcal{K}_{G}$). This rich algebraic structure ensures abundance. In Section \ref{geometry}, we present some classical examples of irreducible subgroups $G$ of $\On$ and their $G$-invariant convex bodies.

Irreducibility has profound connections with the Minkowski-type problem.  Surprisingly,  for any irreducible group $G\subset\On$, the $L_p$ dual Minkowski problem admits a solution among the $G$-invariant convex bodies for all $q, p \in \mathbb{R}$. We thus  resolve the existence problem via variational methods, which rely on the geometric characterization of irreducible groups.

			\begin{theorem}\label{1.2}
			Let $q\in \mathbb{R}$, $p \in \mathbb{R}$, $G$ be an irreducible subgroup of $\mathrm{O}(n)$ with $n \geq 2$, and $Q$ be a $G$-invariant star body in $\mathbb{R}^n$. For any non-zero finite Borel measure $\mu$ on $S^{n-1}$, $\mu$ is $G$-invariant if and only if there exists a $G$-invariant convex body $K$ in $\mathbb{R}^n$ such that
			\[
			\mu = \widetilde{C}_{p,q}(K, Q, \cdot) \quad \text{when  $p \neq q$},
			\]
			and
			\[
			\mu = \lambda \widetilde{C}_{p,q}(K, Q, \cdot) \quad \text{for some $\lambda>0$ when $p=q$}.
			\]
		\end{theorem}

Note that, as mentioned earlier, this solution $K$ is not necessarily origin-symmetric.
We emphasize that the above theorem contains many interesting special  cases.		When $q = n$, the $L_p$ dual Minkowski problem reduces to the $L_p$ Minkowski problem — a case largely open particularly for $p \leq 0$, which encompasses both the logarithmic Minkowski problem and the centro-affine Minkowski problem.   Here we present a complete  solution to the existence part under the group-invariant assumption. It follows from Theorem \ref{1.2} that:

\begin{corollary}
		Let  $G$ be an irreducible subgroup of $\mathrm{O}(n)$ with $n \geq 2$, $p\in\R$.  For any non-zero finite Borel measure $\mu$ on $S^{n-1}$, $\mu$ is $G$-invariant if and only if there exists a $G$-invariant convex body $K$ in $\mathbb{R}^n$ such that
	\[
	\mu = S_{p}(K,  \cdot) \quad \text{when  $p \neq n$},
	\]
	and
	\[
	\mu = \lambda S_{n}(K, \cdot) \quad \text{for some $\lambda>0$}.
	\]
\end{corollary}

When $p=0$, the $\Lp$ dual Minkowski problem reduces to the dual Minkowski problem. Applying Theorem \ref{1.2} with $p=0$ and $Q=B$ yields:

\begin{corollary}
	Let $q \neq 0$,  $G$ be an irreducible subgroup of $\mathrm{O}(n)$ with $n \geq 2$.  For any non-zero finite Borel measure $\mu$ on $S^{n-1}$, $\mu$ is $G$-invariant if and only if there exists a $G$-invariant convex body $K$ in $\mathbb{R}^n$ such that
	\[
	\mu = \widetilde{C}_{q}(K,   \cdot).
	\]
\end{corollary}

When $q=0$, the $(p,0)$-th dual curvature measure, after a slight modification, becomes the $L_p$ integral curvature measure $J_p(K, \cdot)$. Applying Theorem \ref{1.2} with $q=0$ and $Q=B$, and making appropriate adjustments, we obtain the existence of solutions to the $L_p$ Aleksandrov problem.

\begin{corollary}
	Let  $G$ be an  irreducible  subgroup of $\mathrm{O}(n)$ with $n \geq 2$, and let $\mu$ be a non-zero finite Borel measure  on $S^{n-1}$.
	\begin{enumerate}
		\item If $p\neq 0$, 	$\mu$ is $G$-invariant if and only if there exists a $G$-invariant convex body $K$ in $\mathbb{R}^n$ such that
		$\mu=J_{p}(K, \cdot)$.
		\item $\mu$ is $G$-invariant and $| \mu|=nV(B)$ if and only if there exists a $G$-invariant convex body $K$ in $\mathbb{R}^n$ such that
		$\mu$ is the integral curvature of $K$.
	\end{enumerate}
\end{corollary}

\vspace{1\baselineskip} 
	\section{Preliminaries}
	Let $\left( \mathbb{R}^n, \langle \cdot, \cdot \rangle\right) $ denote the $n$-dimensional Euclidean space  with the standard inner product.
	
	Let $S^{n-1}$ represent the unit sphere and $B$ the closed unit ball centered at the origin in $\mathbb{R}^n$.
	Let $B(r)=rB$ be a closed ball with radius $r$ in $\Rn$ and centered at the origin.  	Let $\mathcal{B}^n$ denote the class of Euclidean balls centered at the origin in $\mathbb{R}^n$.
	
	A \textit{convex body} is a compact convex subset of $ \mathbb{R}^{n} $ with non-empty interior. The set of convex bodies in  $\mathbb{R}^{n} $ is denoted by $ \mathcal{K}^{n}$,  the set of origin-symmetric convex bodies in  $\mathbb{R}^{n} $ is denoted by $ \mathcal{K}_{e}^{n}$, and $ \mathcal{K}_{o}^{n} $ denotes the set of convex bodies in $\mathbb{R}^{n}$ that contain the origin in their interior.
	
	The \textit{support function} of a compact convex set $ K $ is given by 
	$$ h_{K}(x)=\max _{y\in K} \la y , x \ra, \quad \text{for $x\in \Rn$}.$$
	
	A set $K$ is \textit{star-shaped} with respect to a point $x$ if every line through $x$ that meets $K$ does so in a (possibly degenerate) closed line segment.
	 If $K$ is a compact star-shaped set with respect to the origin, then for all $u\in S^{n-1}$, its \textit{radial  function} $ \rho_K(\cdot )$   is defined by (see \cite{Gaedner gtbook,schneiderbook2014})
	\[
	\rho_K(u)=\max\{c:cu\in K\},
	\]
	if the radial function of $K$ is continuous, then $K$ is called a \textit{star body}. We define the set of star bodies $\mathcal{S}^{n}_{o}$ in $\Rn$ whose radial functions are positive and continuous. The radial function is  homogeneous of degree $-1$, that is,
	\begin{align}\label{rad_fun_homo}
		\rho_K(cu)=c^{-1}\rho_K(u), \quad c>0.
	\end{align}
	For an invertible linear transformation $\phi$, the support function $h_K$ and the radial function $\rho_K$ transform as follows:
	\begin{align}
		h_{\phi K}(y) &= h_K(\phi^T y), \label{eq:support-transform} \\
		\rho_{\phi K}(y) &= \rho_K(\phi^{-1} y). \label{eq:radial-transform}
	\end{align}
	
The polar body $ K^{*} $ of $ K\in \mathcal{K}_{o}^{n} $ is defined
by
\begin{equation}
	K^{*}=\left\{x \in \mathbb{R}^{n}: \la x, y \ra \leq 1 \quad \text { for all } y \in K\right\}.
\end{equation}
	
	For a convex body $K \in \mathcal{K}_o^n$ and a Borel set $\eta \subset S^{n-1}$, define
	\[
	\alpha_K^*(\eta) =  \big\{u \in S^{n-1} : \rho_K(u)u \in \nu_K^{-1}(\eta)\big\},
	\]
	where  $\nu_{K}$ is the outer unit normal vector  of $K$.
	
		The $q$-th dual mixed volume for star bodies $K, Q \in \mathcal{S}_{o}^{n}$ and $q \in \mathbb{R}$ is given by
	\[
	\widetilde{V}_{q}(K, Q) = \frac{1}{n} \int_{S^{n-1}} \rho_{K}^{q}(u)\rho_{Q}^{n-q}(u)  du.
	\]   
	 When $Q=B$, for each $q = 1, \ldots, n$,  $\widetilde{V}_{q}(K, B)$ coincides exactly with the $(n - q)$-th dual quermassintegral $\widetilde{W}_{n-q}(K)$: 
	 \[
	 \widetilde{W}_{n-q}(K) = \frac{\omega_n}{\omega_q} \int_{G(n,q)} \mathcal{H}^q(K\cap\xi)\,d\xi,
	 \]
	 where the integral is taken with respect to the Haar measure on the Grassmannian $G(n,q)$ of all $q$-dimensional subspaces of $\mathbb{R}^n$. Here $\mathcal{H}^{q}$ denotes the $q$-dimensional Hausdorff measure and $\omega_q$ denotes the $q$-dimensional volume of the  unit ball in $\mathbb{R}^q$.

	For convex bodies $K \in \mathcal{K}_{o}^{n}$, the variation of the $q$-th dual mixed volume yields the $q$-th dual curvature measure with respect to $Q\in \mathcal{S}_{o}^{n}$. This Borel measure on $S^{n-1}$, introduced by Lutwak-Yang-Zhang \cite{lyz2018} (and by Huang-LYZ \cite{HLYZdual} for $Q = B$), is defined as
	\begin{equation}\label{defdual}
	\widetilde{C}_{q}(K, Q, \eta) = \frac{1}{n} \int_{\alpha_{K}^{*}(\eta)} \rho_{K}^{q}(u)\rho_{Q}^{n-q}(u)  du
	\end{equation}
	for every Borel set $\eta\subset\Sn$.
	
		The set of continuous positive functions on the sphere $ S^{n-1} $ will be denoted by $ C^{+}(S^{n-1}) $. For each $ f\in C^{+}(S^{n-1}) $, the Wulff shape $ [f] $ generated by $ f $ is a convex body defined by
	\begin{equation}\label{Wulff shape}
		[f]=\left\{x\in \mathbb{R}^{n}: \la x, v \ra \le f(v),\; \text{for all}\; v\in S^{n-1}\right\}.
	\end{equation}

	The variational formula for the $q$-th dual mixed volume is established in \cite{HLYZdual,lyz2018}. Specifically, for $q \neq 0$, $K \in \mathcal{K}_o^n$, $Q \in \mathcal{S}_o^n$, and a continuous function $f : S^{n-1} \to \mathbb{R}$, 
	\begin{equation}\label{dual vari}
			\lim_{t \to 0} \frac{\widetilde{V}_q([ h_K + tf], Q) - \widetilde{V}_q(K, Q)}{t} = q \int_{S^{n-1}} \frac{f(v)}{h_K(v)} \, d\widetilde{C}_q(K, Q, v),
	\end{equation}
	where $|t|$ is sufficiently small.
	
	When $q = 0$, the $0$-th dual curvature measure, also known as the integral curvature measure, is the variational derivative of the \textit{dual mixed entropy}. The \textit{dual mixed entropy} $\widetilde{E}(K, Q)$ of star bodies $K, Q \in \mathcal{S}_o^n$ is defined by \cite{lyz2018}
	\begin{equation}\label{entropy def} 
	\widetilde{E}(K, Q) = \frac{1}{n} \int_{S^{n-1}} \log \left( \frac{\rho_K(u)}{\rho_Q(u)} \right) \rho_Q^n(u)  du,
	\end{equation}
	where $\rho_K$ and $\rho_Q$ denote the radial functions of $K$ and $Q$, respectively.  For $K \in \mathcal{K}_o^n$, $Q \in \mathcal{S}_o^n$, and a continuous function $f : S^{n-1} \to \mathbb{R}$, the following variational formula holds \cite{lyz2018}:
	\begin{equation}\label{entropy vari}
	\lim_{t \to 0} \frac{\widetilde{E}([h_{K}+tf],Q) - \widetilde{E}(K,Q)}{t} = \int_{S^{n-1}} \frac{f(v)}{h_K(v)}  d\widetilde{C}_0(K,Q,v).
	\end{equation}
	
	From Lemma 5.1 in \cite{lyz2018}, for $q \neq 0$ and a bounded Borel function $f : S^{n-1} \to \mathbb{R}$, we have:
	\begin{equation}
	\int_{S^{n-1}} f(v) \, d\widetilde{C}_q(K, Q, v) = \frac{1}{n} \int_{\partial K} f(\nu_K(x)) \langle \nu_K(x), x \rangle \rho_Q^{n-q}(x) \, d\mathcal{H}^{n-1}(x).
	\end{equation}

The $(p,q)$-th dual curvature measure, introduced by Lutwak-Yang-Zhang \cite{lyz2018}, unifies $L_p$ surface area measures, $L_p$ integral curvature measures, and dual curvature measures. For $p, q \in \mathbb{R}$, $Q \in \mathcal{S}_o^n$, and $K \in \mathcal{K}_o^n$, the $(p,q)$-th dual curvature measure  $\widetilde{C}_{p,q}(K, Q, \cdot)$ is defined by:
		\begin{equation}\label{pqdef}
	d\widetilde{C}_{p,q}(K, Q, \cdot) = h_K^{-p} \, d\widetilde{C}_q(K, Q, \cdot).
	\end{equation}
		Three important special cases are:
	\begin{itemize}
		\item The $L_p$ surface area measure: $S_p(K, \cdot) = n\widetilde{C}_{p,n}(K, B, \cdot)$.
		\item The $L_p$ integral curvature: $J_p(K, \cdot) = n\widetilde{C}_{p,0}(K^*, B, \cdot)$.
		\item The dual curvature measure: $\widetilde{C}_{q}(K,  \cdot)=\widetilde{C}_{0,q}(K, B, \cdot)$.
	\end{itemize}
 As shown in Theorem 6.8 of \cite{lyz2018}, for $\phi \in \text{SL}(n, \mathbb{R})$, the following transformation formula holds:
\begin{equation}\label{measure trans}
	\int_{S^{n-1}} f(v) \, d\widetilde{C}_{p,q}(\phi K, \phi Q, v) = \int_{S^{n-1}} |\phi^{-T}v|^{p} f \left( \frac{\phi^{-T}v}{|\phi^{-T}v|} \right) \, d\widetilde{C}_{p,q}(K, Q, v),
\end{equation}
where $\phi^{-T}$ represents the transpose of the inverse of $\phi$.	
	
	The continuity of the dual mixed volumes and that of the dual curvature measures  are stated as follows:

	\begin{lemma}[\cite{lyz2018}]\label{continuity}
		For $p,q \in \mathbb{R}$ and $Q \in \mathcal{S}^{n}_o$, if $K_m \in \mathcal{K}^{n}_{o}$ with $K_m \to K \in \mathcal{K}^{n}_{o}$, then
		$\widetilde{V}_q(K_m,Q) \to \widetilde{V}_q(K,Q)$, and $\widetilde{C}_{p,q}(K_m,Q,\cdot) \to \widetilde{C}_{p,q}(K,Q,\cdot)$ weakly.
	\end{lemma}
	
	Another important class of dual Minkowski problems is the  chord Minkowski problem, see \cite{guoxizhao,LYZX chord,li chord,xyzz2023}.
	
	For a subgroup $G$ of $\On$,  a convex body $K \subset \mathbb{R}^n$ is \textit{$G$-invariant} if it satisfies
	\[
	gK= K \quad \text{for all} \quad g \in G.
	\]
	A function $f$ defined on $\Sn$ is \textit{$G$-invariant} if
	\[
	f(gx) = f(x) \quad \text{for all} \quad g \in G, x\in \Sn.
	\]
	A Borel measure $\mu$ on $S^{n-1}$ is \textit{$G$-invariant} if
	\[
	\mu(gE) = \mu(E) \quad \text{for all}\quad g \in G \text{ and any Borel set } E \subset S^{n-1}.
	\]

	We say a group $G$ is irreducible if its action on $\mathbb{R}^n$ is irreducible  (i.e., the only subspaces $V \subset \mathbb{R}^n$ satisfying $gV = V$ for all $g \in G$ are $V = \{0\}$ and $V = \mathbb{R}^n$).
	If $G\subset \On$ is irreducible, it has no non-zero fixed points ($n\ge2$). If there exists a non-zero fixed point $x_1 \in \mathbb{R}^n$ satisfying $gx_1 = x_1$ for all $g \in G$, then the line $\mathbb{R}x_1$ forms a nontrivial $G$-invariant subspace, contradicting the irreducibility. Since the centroid  of a convex body is equivariant under linear transformations,  for a $G$-invariant convex body $K$, its centroid is  at the origin.

		Let $G $ be any subgroup of $\On$.  Then $\mathbb{R}^n$ decomposes into an orthogonal direct sum of $G$-invariant subspaces:
		\[
		\mathbb{R}^n = V_1 \oplus V_2 \oplus \cdots \oplus V_k,
		\]
		and	the action of $G$ is irreducible on each $V_i$.  $G$ is irreducible if and only if $k=1$.
		If $\dim V_{i} \geq 2$ for all $i$, $G$ has no non-zero fixed points.

	\vspace{1\baselineskip} 
\section{Classification of Group-Invariant Convex Bodies}\label{classification}
	
	We recall some standard notation. Let $\mathcal{K}^n$ denote the set of convex bodies in $\mathbb{R}^n$, 
	$\mathcal{K}_o^n$ the set of convex bodies in $\mathbb{R}^n$ containing the origin in their interior, 
	and $\mathcal{K}_e^n$ the set of origin-symmetric convex bodies in $\mathbb{R}^n$. 
	Let $\mathcal{B}^n$ denote the class of Euclidean balls centered at the origin in $\mathbb{R}^n$. 
	The purpose of this section is to reveal connections between group-invariant convex bodies and these four classes of convex bodies. 

	We endow the orthogonal group $\On$ with the subspace topology induced by the inclusion $\On \subset \mathbb{R}^{n^2}$. 
	Specifically, we identify each matrix $U = (a_{ij}) \in \On$ with a point in $\mathbb{R}^{n^2}$ by considering its $n^2$ entries as coordinates. 
The group $\On$ is compact in $\mathbb{R}^{n^2}$ because it is both bounded and closed. It is bounded since for any $U = (a_{ij}) \in \On$, we have $|a_{ij}| \leq 1$ for all $i, j$, and it is closed because the condition $U^T U = I$ defines a closed set.
	Consequently, any closed subgroup  of $\On$ is also compact.  
	
		Given a subgroup $G$ of $\On$, 
	we denote by $\KG$ the collection of all $G$-invariant convex bodies, i.e.,
	\begin{equation*}
		\KG := \left\{ K \in \mathcal{K}^n : gK = K \text{ for all } g \in G \right\}.
	\end{equation*}
	
	First, it is well-known that  when $G = \On$ or $G=\SOn$, the only $G$-invariant convex bodies 
	are Euclidean balls centered at the origin. For any subgroup $G \subset \On$, 
	balls centered at the origin are always  $G$-invariant. 
	We  prove that if balls are the only $G$-invariant convex bodies, 
	then $G$ acts transitively on the unit sphere,	where the \textit{transitive action on $\Sn$} means that for any 
	$u, v \in \Sn$, there exists $g \in G$ such that $gu = v$.

	\begin{proposition}\label{G=ball}
		Let \( G \) be a closed subgroup of the orthogonal group \( \On \). The following are equivalent:
		\begin{enumerate}
			\item Every \( G \)-invariant convex body in \( \R^n \) is a Euclidean ball centered at the origin, i.e., $\KG=\mathcal{B}^{n}$.
			\item \( G \) acts transitively on the unit sphere \( \Sn \).  
		\end{enumerate}
	\end{proposition}
	
	\begin{proof}
		\medskip\noindent
		\textbf{(2) \(\Rightarrow\) (1):} 
		Assume \( G \) acts transitively on \( \Sn \). Let \( K \) be a \( G \)-invariant convex body. Fix a  vector $v \in S^{n-1}$, \( h_K(v)=r \). For any $u \in S^{n-1}$, 
		the transitivity of the $G$-action guarantees that there exists $g \in G$ 
		satisfying $gv = u$. By \eqref{eq:support-transform}, the support function  satisfies:
		\[
		h_K(u)=h_K(gv) =  h_{g^{-1}K}(v)=h_K(v)=r.
		\]
		This implies \( h_K \) is constant on \( \Sn \).  Thus, \( K \) is the Euclidean ball of radius \( r \).
		
		\medskip\noindent
		\textbf{(1) \(\Rightarrow\) (2):} 
		Suppose \( G \) does not act transitively on \( \Sn \), then there exist $u, v \in \Sn$ with $u \notin Gv=\{gv: g\in G\}$. 
		We construct a $G$-invariant convex body that is not a Euclidean ball.
		Define $K = \operatorname{conv}\left( B(\frac{1}{2}), Gv \right)$, 
		where $B(r)$ is the closed ball of radius $r$ centered at the origin.
		
		We have $u \notin K$. Otherwise, if \( u \in K \), there exist 
		\( g_1, \dots, g_m \in G \), 
		\( x_{m+1}, \dots, x_k \in B(\frac{1}{2}) \), 
		and weights \( \lambda_1, \dots, \lambda_k \geq 0 \) with \( \sum_{i=1}^k \lambda_i = 1 \) 
		such that:
		\[
		u = \lambda_1 g_1 v + \cdots + \lambda_m g_m v + \lambda_{m+1} x_{m+1} + \cdots + \lambda_k x_k.
		\]
		 We may assume \(\lambda_{1},\dots,\lambda_{m}>0\).  Since \(g_{i}\in G\subset \On\), \(|g_{i}v|=|v|=1\), and \(|x_{j}|\le \frac{1}{2}\), then
		\begin{align*}
			1=|u|&=|\lambda_{1}g_{1}v+\cdots+\lambda_{m}g_{m}v+\lambda_{m+1}x_{m+1}+\cdots+\lambda_{k}x_{k}|\\
			&\le \lambda_{1}|g_{1}v|+\cdots+\lambda_{m}|g_{m}v|+\lambda_{m+1}|x_{m+1}|+\cdots+\lambda_{k}|x_{k}|\\
			&\le \lambda_{1}+\cdots+\lambda_{m}+\frac{1}{2}\lambda_{m+1}+\cdots+\frac{1}{2}\lambda_{k}\\
			&\le\lambda_{1}+\cdots+\lambda_{m}+\lambda_{m+1}+\cdots+\lambda_{k}=1.
		\end{align*}
		The equality holds if and only if $g_1v =\cdots = g_mv\triangleq w\in \Sn$ and \(\lambda_{m+1}=\cdots=\lambda_{k}=0\). Then $u=w=g_1v =\cdots = g_mv$, which contradicts the fact that $u \notin Gv$. Thus, \(v\in K\) but \(u\notin K\). Since $|u|=|v|=1$, $K$ cannot be a Euclidean ball centered at the origin.
		
		Since $G$ is compact and the group action $(g,v) \mapsto gv$ is continuous, then
		the orbit $Gv$ is compact, $K$ is a compact convex set. Furthermore,  $B(\frac{1}{2})\subset K$. Hence $K$ is a convex body. For any $g\in G$ and any $$y=\lambda_{1}g_{1}v+\cdots+\lambda_{m}g_{m}v+\lambda_{m+1}x_{m+1}+\cdots+\lambda_{k}x_{k} \in K,$$
		we have
		 $$gy=\lambda_{1}gg_{1}v+\cdots+\lambda_{m}gg_{m}v+\lambda_{m+1}gx_{m+1}+\cdots+\lambda_{k}gx_{k} \in K$$
		 because $gg_{i}\in G$ and $|gx_{j}|=|x_{j}|\le\frac{1}{2}$. Therefore, $gK \subset K$. Since $g$ was arbitrary, replacing $g$ by $g^{-1}$ gives $g^{-1}K \subset K$, which implies $K \subset gK$. Consequently, $gK = K$ for all $g \in G$. Then $K$ is a \( G \)-invariant convex body. But $K$ is not a ball, which contradicts $(1)$.

	\end{proof}
	
	Therefore, if the class of $G$-invariant convex bodies coincides precisely with the set of balls, then $G$ must  be an infinite group when $n\ge 2$.

	In the study of the $L_{p}$ dual Minkowski problem, origin-symmetric convex bodies form an important class of objects. When $G = \{\pm I\}$ where $I$ denotes the identity map, it is clear that the class of $G$-invariant convex bodies coincides precisely with $\mathcal{K}_e^n$.
	
	This observation naturally leads to the following question: Under what conditions does the class of $G$-invariant convex bodies exactly equal $\mathcal{K}_e^n$, contain $\mathcal{K}_e^n$, or be contained in $\mathcal{K}_e^n$?

	\begin{proposition}\label{G=osym}
		Let $G$ be a subgroup of the orthogonal group $\On$. 
		The following conditions are equivalent:
		\begin{enumerate}
			\item $\KG=\mathcal{K}_{e}^{n}$.
			\item $G = \{\pm I\}$.
		\end{enumerate}
	\end{proposition}
	
	\begin{proof}
		We prove both directions of the equivalence.
		
		\noindent\textbf{(2 $\Rightarrow$ 1):} 
		Assume $G = \{\pm I\}$. 
		If $K$ is $G$-invariant, then for $g = -I \in G$, we have $(-I)K = -K = K$. 
			Thus $K$ is origin-symmetric.
		 If $K$ is origin-symmetric ($K = -K$), then
				$IK = K$, 
				$(-I)K = -K = K$,
			so $K$ is $G$-invariant.
		Hence the $G$-invariant convex bodies are precisely the origin-symmetric convex bodies.
		
		\noindent\textbf{(1 $\Rightarrow$ 2):} 
		Let $T: \mathbb{R}^n \to \mathbb{R}^n$ be a linear transformation such that for every $v \in \mathbb{R}^n$, either $Tv = v$ or $Tv = -v$. Then $T = \pm I$. Indeed, let $\{e_1,\ldots,e_n\}$ be an orthonormal basis. For each $i$, we have
		\begin{equation}\label{1}
			Te_i = \lambda_i e_i \quad \text{with } \lambda_i = \pm1.
		\end{equation}
	  For $i \neq j$, $T(e_{i}+e_{j})=\lambda_i e_i+\lambda_j e_j=\mu(e_{i}+e_{j})$
		for some $\mu = \pm 1$, implying $\lambda_i = \lambda_j$. Thus all $\lambda_i$ are equal. By \eqref{1} we have $T=\pm I$. 
	
		Assume condition (1) holds. We show that $G = \{\pm I\}$.
	If there exists $g \in G$ with $g \neq \pm I$,  we construct a convex body that is origin-symmetric but not $G$-invariant.
		Since $g \neq \pm I$, there exists a unit vector $v \in \mathbb{R}^n$ such that $gv \neq \pm v$. 
		Assume $v = e_1$ after a change of coordinates.
		
		Consider the ellipsoid (an origin-symmetric convex body)
		\begin{equation}\label{elliK}
		K = \{ x \in \mathbb{R}^n : x^T A x \leq 1 \},
		\end{equation}
		where $A = \operatorname{diag}(1, 2, \dots, 2)$ is a positive definite matrix. 
		Under the action of $g$, we have:
		\begin{align}\label{elligK}
		gK &= \{ gx \in \mathbb{R}^n : x^T A x \leq 1 \}\notag\\
		&=\{ y \in \mathbb{R}^n : (g^{-1}y)^T A (g^{-1}y) \leq 1 \}\notag\\
		&=\{ y \in \mathbb{R}^n : y^T (g A g^T) y \leq 1 \}.
		\end{align}
		
		Now compute the quadratic forms:
		\[
		e_1^T A e_1 = 1,
		\]
		\[
		e_1^T (g^{T} A g) e_1 = (g e_1)^T A (g e_1) = w^T A w, \quad \text{where} \quad w = g e_1.
		\]
		Since $g$ is orthogonal, $w$ is a unit vector. As $g e_1 \neq \pm e_1$, we have $w \neq \pm e_1$. 
	Let $w = (w_1, w_2, \dots, w_n)$ with $w_1^2 < 1$ and $\sum_{i=2}^n w_i^2 = 1 - w_1^2 > 0$. Then
		\[
		w^T A w =  w_1^2 + 2 \sum_{i=2}^n w_i^2 =  w_1^2 + 2 (1 - w_1^2)=2 - w_1^2 > 1  = e_1^T A e_1.
		\]
	 Thus $g^{T} A g \neq A$, since $g^{T}=g^{-1}$, so $g A g^{T} \neq A$. Then we have  $gK \neq K$ by \eqref{elliK} and \eqref{elligK}, contradicting 
	 $(1)$.
		Therefore, if $g\in G$, $g =  I$ or $g=-I$.
		
		 So we have $G \subset \{\pm I\}$. 
		If $G = \{I\}$, then all convex bodies (including non-origin-symmetric ones) are $G$-invariant, 
		contradicting  $(1)$. Thus $G = \{\pm I\}$.
	\end{proof}

	In Proposition \ref{G=osym}, we established that for any $g \neq \pm I$, there exists an origin-symmetric convex body $K$ such that $gK \neq K$. This leads to the following lemma:

\begin{lemma}

	Let $G$ be a subgroup of  $\On$. Then $ \mathcal{K}^{n}_{e}\subset \KG$ if and only if $G=\{I\}$ or $G=\{\pm I\}$.

\end{lemma}
	
If $-I \in G$, then the $G$-invariant convex bodies are obviously origin-symmetric. The converse, however, is not true. For example, when $n$ is odd, $-I \notin \mathrm{SO}(n)$, but the $\mathrm{SO}(n)$-invariant convex bodies are only balls, which are naturally origin-symmetric. In fact, we have the following equivalence.

		\begin{proposition} \label{Gsub osym}
			Let \(G\) be a closed subgroup of \(\On\). The following conditions are equivalent:
			\begin{enumerate}
				\item $\KG \subset \mathcal{K}^{n}_{e}$.
				\item For every  \(x \in \Rn \),  \(-x \in Gx\).
			\end{enumerate}
		\end{proposition}

	\begin{proof}
			\noindent\textbf{(2 $\Rightarrow$ 1):} 
		Assume $(2)$ holds, and let \(K\) be a \(G\)-invariant convex body. To show \(K = -K\), take any \(x \in K\).  Since $K$ is $G$-invariant, \(gx \in K\) for all \(g \in G\).   By the  condition (2), \(-x \in Gx \subset K\).  Thus, \(K = -K\).
		
		\noindent\textbf{(1 $\Rightarrow$ 2):} 
		Assume there exists \(x \in \Rn\setminus \{0\}\) such that $-x \notin Gx $. Using the same construction as in Proposition \ref{G=ball}, choose $0 < r < |x|$ and define  
		$K = \operatorname{conv}(B(r), Gx)$.  Since $|-x|=|gx|$ for all $g\in G$, a similar argument to Proposition \ref{G=ball} shows that
		 $-x \notin K$.  But $x\in K$, so $K\neq -K$.
		Since $G$ is compact,  
		the orbit $Gx$ is compact, and thus $K$ is a compact convex set. Furthermore,  $B(r)\subset K$, so $K$ is a convex body. $K$ is clearly $G$-invariant, but $K \notin \mathcal{K}^{n}_{e}$, which contradicts $
		(1)$.
		
	\end{proof}

	But when $G$ is finite, we have
		
		\begin{proposition}\label{finite -i}
			Let \( G \) be a finite subgroup of  \( \mathrm{O}(n) \). 
		Then	$\KG \subset \mathcal{K}^{n}_{e}$ if and only if 
	\(-I\in G\).	
		\end{proposition}
		
		\begin{proof}
		A	finite group $G$ is naturally closed. From Proposition \ref{Gsub osym}, it suffices to prove that for every  \( x \in \mathbb{R}^n \),  \(-x\in Gx\) if and only if \(-I\in G\).
		
			If \(-I\in G\), $-x=-Ix$ is naturally contained in $Gx$.			
Conversely,	suppose \(-x\in Gx\) for any \( x \in \mathbb{R}^n \),  but \(-I \notin G\). For each  \( g \in G \), define the subspace
			\[
			V_g := \{ x \in \mathbb{R}^n \mid g  x = -x \}.
			\]
			Since \( g \) is  linear, \( V_g \) is  a linear subspace of \( \mathbb{R}^n \). \( V_g \neq \mathbb{R}^n \) for any \( g \in G \), because  \( V_g = \mathbb{R}^n \) would imply \( g = -I \), contradicting the assumption that \(-I \notin G\).
				
			By the given condition, for every \( x \in \mathbb{R}^n \), there exists some \( g_x \in G \) such that \( g_{x}  x = -x \). This implies that \( x \in V_{g_x} \). Therefore,
			\[
			\mathbb{R}^n = \bigcup_{g \in G} V_g.
			\]
			Since \( G \) is finite,  a finite union of proper subspaces cannot cover the entire space. This leads to a contradiction.
		\end{proof}

	Another fundamental class in convex geometry is $\mathcal{K}_o^n$, consisting of convex bodies containing the origin in their interior. We now examine the relationship between $G$-invariant convex bodies and this class.
	
	\begin{proposition} \label{int1}
		Let \(G\) be a  subgroup of \(\On\). The following conditions are equivalent:
		\begin{enumerate}
			\item $\KG \subset \mathcal{K}^{n}_{o}$.
			\item $G$ has no non-zero fixed points in $\mathbb{R}^n$, i.e., if $gx=x$ for any $g\in G$, then $x=0$.
		\end{enumerate}
	\end{proposition}
	
		\begin{proof}
		\noindent\textbf{(1 $\Rightarrow$ 2):} 
		Suppose, to the contrary, that there exists a non-zero vector $w \in \mathbb{R}^n $ 
		such that $g w = w$ for all $g \in G$. Let 
		\[
		K = \left\{ x \in \mathbb{R}^n : |x| \leq 1 \text{ and } \langle x, w \rangle \leq 0 \right\}.
		\]
		Then $K$ is a convex body as the intersection of the closed unit ball and a half-space. 
		Moreover, $K$ is $G$-invariant: for any $g \in G\subset \On$ and $x \in K$, we have $|gx| =|x| \leq 1$ and
		\[
		\langle gx, w \rangle = \langle gx, gw \rangle = \langle x, w \rangle \leq 0,
		\]
	so $gx\in K$. Thus $gK = K$. However, the origin is not in the interior of $K$, 
		because for any $t > 0$, the point $t w$ satisfies $\langle t w, w \rangle = t |w|^{2} > 0$, so $t w \notin K$. 
		This contradicts $\KG \subset \mathcal{K}^{n}_{o}$.
		
	\noindent\textbf{(2 $\Rightarrow$ 1):}  
		Assume $G$ has no non-zero fixed points. Let $K \in \mathcal{K}_{G}$ be a $G$-invariant convex body. 
		Since the centroid of a convex body is equivariant with linear transformations, the centroid of $K$ is G-invariant. More precisely, let $T(K)$ denote the centroid of $K$, we have
		$$gT(K)=T(gK)=T(K)$$
		for all $g\in G$. Thus, 
		$T(K)$ is a fixed point of $G$. Since $G$ has no nonzero fixed points, it follows that $T(K) = 0$. Since the centroid of the convex body $K$ is at the origin, we conclude that $K \in \mathcal{K}_o^n$.
	\end{proof}

	For $G = \{I\}$, the $G$-invariant convex bodies constitute the entire collection of convex bodies, which naturally includes $\mathcal{K}_o^n$. Furthermore, we have:

		\begin{proposition} \label{int2}
		Let \(G\) be a  subgroup of \(\On\). The following conditions are equivalent:
		\begin{enumerate}
			\item $\mathcal{K}^{n}_{o} \subset \KG$.
			\item $G=\{I\}$.
		\end{enumerate}
	\end{proposition}
	
	\begin{proof}
			\noindent\textbf{(2 $\Rightarrow$ 1):}
		If	$G = \{I\}$, then $\KG$ is the entire collection of convex bodies. 
		 
		\noindent\textbf{(1 $\Rightarrow$ 2):}
Suppose	$\mathcal{K}^{n}_{o} \subset \KG$ but $G$ is non-trivial. We  construct a convex body $K \in \mathcal{K}^{n}_{o} $ that is not $G$-invariant, contradicting the assumption.
		
		Since $G$ is non-trivial, there exists $g \in G$ with $g \neq I$. Then there exists $v \in \R^n$ such that $gv \neq v$ and $v \neq 0$. Fix $r > 0$ such that $r < |v|$ and define
		\[
		K = \conv\left( B(r) \cup \{v\} \right).
		\]
		Thus, the subsequent proof is  analogous to Proposition \ref{G=ball} and Proposition \ref{Gsub osym}.
		Since $gv\neq v$ and $|gv|=|v|$,
		we have $gv \notin K$. But $v\in K$, so $K$ is not $G$-invariant.
		Since $B(r)\subset K$, so $K\in\mathcal{K}^{n}_{o}$,  which contradicts condition $
		(1)$.
	\end{proof}

Propositions \ref{int1} and \ref{int2} imply that no subgroup $G$ of $\On$ 
satisfies $\mathcal{K}_G = \mathcal{K}_o^n$.  Moreover, if $\KG = \mathcal{K}^{n}$, then necessarily $G = \{I\}$.

	The main results of this section are summarized as follows, 	where $\mathcal{C}$ denotes  the class of convex bodies corresponding to the given columns.

	\begin{table}[h]
		\centering
		\begin{tabular}{|c|c|c|c|c|}
			\hline
			\diagbox[width=5.5em, height=1.8em]{}{$\mathcal{C}$} & $\mathcal{B}^n$ & $\mathcal{K}_e^n$ & $\mathcal{K}_o^n$ & $\mathcal{K}^n$ \\ 
			\hline
			$\KG = \mathcal{C} \iff$ & 
			\begin{minipage}{0.18\textwidth}
				$G$ acts transitively on $S^{n-1}$
			\end{minipage} & 
			\begin{minipage}{0.18\textwidth}
				$G = \{\pm I\}$
			\end{minipage} & 
			\begin{minipage}{0.18\textwidth}
				No such $G$
			\end{minipage} &
			\begin{minipage}{0.18\textwidth}
			$G=\{I\}$
			\end{minipage} \\
			\hline
			$\KG \supset \mathcal{C} \iff$ & 
			\begin{minipage}{0.18\textwidth}
				Always true
			\end{minipage} & 
			\begin{minipage}{0.18\textwidth}
				$G = \{I\}$ or $\{\pm I\}$
			\end{minipage} & 
			\begin{minipage}{0.18\textwidth}
				$G = \{I\}$
			\end{minipage} &
			\begin{minipage}{0.18\textwidth}
				$G=\{I\}$
			\end{minipage} \\
			\hline
			$\KG \subset \mathcal{C} \iff$ & 
			\begin{minipage}{0.18\textwidth}
				$G$ acts transitively on $S^{n-1}$
			\end{minipage} & 
			\begin{minipage}{0.18\textwidth}
				$-x \in Gx$ for all $x$
			\end{minipage} & 
			\begin{minipage}{0.18\textwidth}
				No non-zero fixed points
			\end{minipage} &
			\begin{minipage}{0.18\textwidth}
				Always true
			\end{minipage} \\
			\hline
		\end{tabular}
	\end{table}

	\begin{remark}
		In Propositions \ref{G=ball} and \ref{Gsub osym}, we require $G$ to be a closed subgroup, while other propositions in this section impose no closedness restriction. Without the closedness assumption, Propositions \ref{G=ball} and \ref{Gsub osym} become invalid, since the set $K$ constructed in the proof may not  be a convex body, though they can be corrected with  appropriate modifications. For brevity, we discuss counterexamples and generalized versions without closedness assumptions in  Section \ref{open}. Consequently, we have actually established complete characterizations of the containment relations between the class of $G$-invariant convex bodies and the classes  $\mathcal{B}^{n}$, $\mathcal{K}^n_{e}$, $\mathcal{K}^n_{o}$, and $\mathcal{K}^n$ for arbitrary subgroups $G$ of $\On$.
	\end{remark}

\vspace{1\baselineskip} 
	\section{Geometric Characterization of Irreducible Closed Subgroups}\label{geometry}
	A group $G \subset \On$ is irreducible if the only subspaces $V \subset \mathbb{R}^n$ that are invariant under $G$ (i.e., $gV = V$ for all $g \in G$) are $\{0\}$ and $\mathbb{R}^n$.
	
	In the previous section, we classified  $G$-invariant convex bodies. We now turn to investigating the Minkowski problem. Specifically, when $G$ is an irreducible closed subgroup of $\On$, we study the existence of solutions to the  $G$-invariant $L_{p}$  dual Minkowski problem. The objective of this section is to characterize the geometry of irreducible closed subgroups, which is indispensable for  solving the Minkowski problem using variational methods.

	The following observation is crucial, as it builds a bridge between irreducibility and the Minkowski problem.
	
	\begin{proposition}\label{irr just}
		Let $G$ be a  closed subgroup of $\On$ with $n\ge2$. The following conditions are equivalent:
		\begin{enumerate}
			\item For every  $v \in S^{n-1}$, the orbit $Gv$ is not contained in any closed hemisphere.
		
			\item  $G$  is irreducible (i.e., $\mathbb{R}^n$ has no nontrivial $G$-invariant subspaces).
		\end{enumerate}
	\end{proposition}
	
	\begin{proof}
		\noindent\textbf{(1 $\Rightarrow$ 2):}
		Assume condition $(1)$ holds. Suppose, for contradiction, that $G$ is reducible. Then there exists a nontrivial  subspace $W \subset \mathbb{R}^n$ ($0 < \dim W < n$) that is $G$-invariant ($gW=W$ for all $g\in G$). Since $G \subset \On$, the orthogonal complement $W^\perp$ is also $G$-invariant: If $u\in W^{\perp}$, for any $w\in W$,  since $W$ is $G$-invariant, $g^{-1}w\in W$ for all $g\in G$, then $\langle gu, w \rangle=\langle u, g^{-1}w\rangle=0$. Thus, $gu\in W^{\perp}$.

		Take a unit vector $v \in W$, then $Gv \subset W$ by $G$-invariance. Choose a unit vector $u \in W^\perp$. For all $x \in Gv \subset W$, we have $\langle x, u \rangle = 0$. Thus,
		\[
		Gv \subset \{x \in S^{n-1} : \langle x, u \rangle \geq 0\},
		\]
		contradicting $(1)$. Hence $G$ is irreducible.
		
		\noindent\textbf{(2 $\Rightarrow$ 1):}
		Assume $G$ is irreducible. For any $v\in\Sn$,   irreducibility implies $\text{span}(Gv) = \mathbb{R}^n$; otherwise $\text{span}(Gv)$ would be a nontrivial $G$-invariant subspace.

			Suppose there exists a vector $v \in S^{n-1}$, such that the orbit $Gv$ is  contained in some closed hemisphere, then there exists a  unit vector $u \in S^{n-1}$ such that
			\begin{equation}\label{closed hemi}
		 \langle g v, u \rangle \geq 0 \quad\text{ for all}\quad g \in G.
		 \end{equation}
		 
		 Since $G$ is a compact group, there exists a $G$-invariant probability measure (the Haar measure) $dh$ on $G$.  Let $\bar{v}=\int_G gv dh(g)$. For any $g_{1}\in G$, we have
		 \begin{align*}
		 	g_{1}\bar{v}=&\int_G g_{1}gv dh(g) =\int_G sv dh(g_{1}^{-1}s)\\
		 	=&\int_G sv dh(s)=\bar{v}.
		 \end{align*}
		 In other words, $\bar{v}$ is a fixed point of $G$. 	Since $G$ is irreducible, $n\ge 2$, it has no non-zero fixed points. If there exists a non-zero fixed point $x_1 \in \mathbb{R}^n$ satisfying $gx_1 = x_1$ for all $g \in G$, then the line $\mathbb{R}x_1$ forms a nontrivial $G$-invariant subspace, contradicting the irreducibility. Therefore, 
		 \begin{equation} \label{barv} 
		 	\bar{v}=\int_G gv dh(g)=0.
		 \end{equation}

		By \eqref{closed hemi} and \eqref{barv},     consider the nonnegative continuous function $g \mapsto \langle g  v, u \rangle$. Its integral satisfies:
	\[
	\int_G \langle g  v, u \rangle  dh = \left\langle \int_G g  v  dh, u \right\rangle =\langle \bar{v}, u \rangle = \langle 0, u \rangle = 0.
	\]
		Since $\langle g  v, u \rangle$ is nonnegative, continuous, and the integral vanishes, note that the Haar measure is positive on nonempty open sets. If $\langle g_{0}  v, u \rangle>0$ for some $g_{0}\in G$, then by continuity there exists a neighborhood $U$ of $g_{0}$ such that $\langle g  v, u \rangle>0$ in $U$. This implies $	\int_G \langle g  v, u \rangle  dh >0$,  leading to a contradiction. Thus,  we must have $\langle g  v, u \rangle = 0$ for all $g \in G$. This implies $u \perp \text{span}(Gv) = \mathbb{R}^n$, so $u = 0$, contradicting $|u| = 1$. Hence $Gv$ is not contained in any closed hemisphere, proving $(1)$.
		
	\end{proof}

		From Proposition \ref{irr just}, we observe that when $G$ is irreducible, a $G$-invariant measure $\mu$ is  not concentrated
	on any closed hemisphere.
	
	\begin{corollary}
		Let $G$ be an irreducible closed subgroup of $\On$, $n\ge2$.  If $\mu$ is a $G$-invariant nonzero Borel measure on $\Sn$, then $\mu$ is  not concentrated
		on any closed hemisphere.
	\end{corollary}
	
	\begin{proof}
		Suppose, for contradiction, that $\mu$ is concentrated on a  closed hemisphere $H_u=\{x\in\Sn: \langle x,u \rangle\ge0\}$ for some $u\in\Sn$. Then $\operatorname{supp}(\mu) \subset H_u$.
		Since $\mu$ is $G$-invariant and nonzero, its support $\operatorname{supp}(\mu)$ is a nonempty, closed, $G$-invariant set. Choose any $v \in \operatorname{supp}(\mu)$. By $G$-invariance, the orbit $Gv$ is contained in $\operatorname{supp}(\mu)$, and hence in $H_u$.
		This contradicts Proposition~\ref{irr just}. Therefore, $\mu$ cannot be concentrated on any closed hemisphere.
	\end{proof}
	
This suggests a profound connection between irreducibility and the Minkowski-type problem.

		We need the following lemma.
	\begin{lemma}\label{hemi just}
		Let $G$ be a subgroup of $\On$, $v\in \Sn$. Then
		$Gv$ is not contained in any closed hemisphere  if and only if $0\in \operatorname{int} (\conv (Gv))$.
	\end{lemma}
	
	\begin{proof}
		Let	$H_{u}=\{x\in\Sn: \langle x,u \rangle\ge0\}$ denote the closed hemisphere with normal vector $u$. If  $0\in \operatorname{int} (\conv (Gv))$, for any $u\in\Sn$, there exists $x\in \conv(Gv)$ such that $\langle x,u\rangle<0$ (since $B(\epsilon)\subset \conv (Gv)$ for some $\epsilon$, take $x=-\epsilon u$). Let $x=\sum_{i=1}^{m}\lambda_{i}g_{i}v$ where \(g_{i}\in G\),  \(\lambda_{i}\ge0\) and \(\sum_{i=1}^{m}\lambda_{i}=1\). In other words,
		$\sum_{i=1}^{m} \lambda_{i}\langle g_{i}v, u \rangle <0$. Then there exists $g_{i}$ such that $\langle g_{i}v, u \rangle <0$. Therefore, $Gv\not\subset Hu$.
		
	Conversely,	suppose $Gv\not\subset Hu$ for any $u\in\Sn$. If $0\notin \operatorname{int}\conv(Gv)$, then there exists $u\in\Sn$ such that $\conv(Gv)\subset \{x\in\Rn:\langle x,u\rangle\ge0\}$. In other words, $Gv\subset H_{u}$. This leads to a contradiction.
	\end{proof}
	
	
			To solve the Minkowski problem via variational methods, we require uniform estimates for certain geometric structures associated with irreducible closed subgroups $G\subset \On$. From Lemma~\ref{hemi just} and Proposition~\ref{irr just}, we know that for any  $v \in \Sn$, the convex hull of its orbit $\operatorname{conv}(Gv)$  is a convex body containing the origin in its interior. Furthermore, we have:
	
	\begin{lemma}\label{int estimate}
		Let $G$ be an irreducible closed subgroup of $\On$ with $n\ge2$. There exists $r > 0$, independent of $v$, such that for all $v \in S^{n-1}$, the inclusion
		$B(r)\subset \conv (Gv)$ holds.  
	\end{lemma}

	\begin{proof}
		Assume no such $r$ exists. Then for each $k \in \mathbb{N}$, there exists a unit vector $v_k \in S^{n-1}$ such that:
		\[
		B\left( \frac{1}{k}\right)  \not\subset \operatorname{conv}(Gv_k).
		\]
		Thus there exists $y_k \in B\left( \frac{1}{k}\right) $ with $y_k \notin \operatorname{conv}(Gv_k)$. Since $S^{n-1}$ is compact, $\{v_k\}$ has a convergent subsequence (still denoted by $\{v_k\}$) satisfying $v_k \to v$ for some $v \in S^{n-1}$.	For any $\epsilon > 0$, if $k$ is sufficiently large, we have
		\[
		|v_k - v| < \epsilon \quad \text{and} \quad |g v_k - g v|=|g(v_k - v)|=|v_k - v| < \epsilon \quad \text{for all $g\in G$}.
		\]
		Therefore, for any $x=\sum_{i=1}^{m}\lambda_{i}g_{i}v_{k}\in\conv (Gv_{k})$ where $\lambda_{i}>0$, $g_{i}\in G$ and $\sum_{i=1}^{m} \lambda_{i}=1$, there exists $y=\sum_{i=1}^{m}\lambda_{i}g_{i}v\in \conv (Gv)$ such that
		\begin{align*}
			|x-y|=&|\lambda_{1}(g_{1}v_{k}-g_{1}v)+\cdots+\lambda_{m}(g_{m}v_{k}-g_{m}v)|\\
			\le&\lambda_{1}|g_{1}v_{k}-g_{1}v|+\cdots+\lambda_{m}|g_{m}v_{k}-g_{m}v|\\
			\le&\sum_{i=1}^{m} \lambda_{i}\epsilon =\epsilon.
		\end{align*}
		Similarly, for any $y = \sum_{i=1}^{m} \lambda_i g_i v \in \operatorname{conv}(Gv)$, 
		take $x = \sum_{i=1}^{m} \lambda_i g_i v_k \in \operatorname{conv}(Gv_k)$, then $|x - y| < \epsilon$.
		Then	 the Hausdorff metric satisfies
		$d_H(\operatorname{conv}(Gv_k), \operatorname{conv}(Gv)) \to 0$ as $k \to \infty$.

		By irreducibility of $G$ (Lemma \ref{hemi just} and Proposition \ref{irr just}), $0 \in \operatorname{int}(\operatorname{conv}(Gv))$, so there exists $\delta > 0$ such that:
		\[
		B(\delta) \subset \operatorname{conv}(Gv).
		\]
		Let $C_{k}=\operatorname{conv}(Gv_k)$ and $C=\operatorname{conv}(Gv)$. Since 	$d_H(C_k, C) \to 0$, there exists a large $N$ such that for any $k>N$, the support function satisfies
		\begin{equation}\label{sup}
			\left|h_{C_{k}}(u)-h_{C}(u)\right|<\frac{\delta}{2},\quad \text{for each $u\in\Sn$}.
		\end{equation}
		Since $B(\delta) \subset C$, we have $h_{C}(u)\ge\delta$. Then
		by \eqref{sup},	
		$$h_{C_{k}}(u)>h_{C}(u)-\frac{\delta}{2}\ge \delta-\frac{\delta}{2}=\frac{\delta}{2} \quad \text{for each $u\in\Sn$}.$$
		In other words, 	$B(\frac{\delta}{2}) \subset C_{k}= \operatorname{conv}(Gv_{k})$. Let $k>\max\{N,\frac{2}{\delta}\}$, we have:
		\[
		y_{k}\in B\left( \frac{1}{k}\right)  \subset B\left( \frac{\delta}{2}\right)   \subset \operatorname{conv}(Gv_k),
		\]
		which implies $y_k \in \operatorname{conv}(Gv_k)$, contradicting $y_k \notin \operatorname{conv}(Gv_k)$. Thus the lemma holds.
	\end{proof}
	
	We need the following lemma, which is geometrically evident.

\begin{lemma}\label{Sbound}
	Let \( a > 0 \) and \( S \subset \Sn \). If the vectors in \( S \) are not contained in any closed hemisphere, then
		\[
	K = \bigcap_{u \in S} \{ x \in \mathbb{R}^n : \langle x, u \rangle \leq a \}
	\]
	is a compact convex set.
\end{lemma}
	
	\begin{proof}
	Clearly, $K$ is closed and convex. If $K$ is unbounded, then there exists a sequence $\{x_k\} \subset K$ such that 
	\[
	\lim_{k \to \infty} |x_k| = \infty.
	\]
	Consider the sequence $y_k = \frac{x_k}{|x_k|}$. Since $y_k \in \Sn$ for all $k$,  there exists a convergent subsequence $\{y_{k_j}\}$ such that 
	\[
	\lim_{j \to \infty} y_{k_j} = v \quad \text{for some} \quad v \in S^{n-1}.
	\]
	
	For each $x_{k_j} \in K$ and every $u \in S$, the definition of $K$ implies 
	\[
	\langle x_{k_j}, u \rangle \leq a.
	\]
	Dividing both sides by $|x_{k_j}| > 0$, we obtain 
	\[
	\langle y_{k_j}, u \rangle \leq \frac{a}{|x_{k_j}|}.
	\]
	Taking the limit as $j \to \infty$, the left side converges to $\langle v, u \rangle$, and the right side converges to $0$ (since $|x_{k_j}| \to \infty$). Thus, for every $u \in S$,
	\[
	\langle v, u \rangle \leq 0.
	\]
	This implies that $S$ is contained in the closed hemisphere $\{ x \in S^{n-1} : \langle v, x \rangle \leq 0 \}$. This leads to a contradiction. Therefore, $K$ is compact.
	\end{proof}
	
	Thus,  when the group $G$ is irreducible, by setting $S = Gv$, we obtain:
	\begin{lemma}\label{Kv convex}
			Let $G$ be an irreducible closed subgroup of $\On$, $n\ge2$. For every  $v \in S^{n-1}$ and \( a > 0 \), the set
		\[
		K(v) = \bigcap_{u \in Gv} \{ x \in \mathbb{R}^n : \langle x, u \rangle \leq a \}
		\]
		is a convex body.
	\end{lemma}
	
	\begin{proof}
	By	Lemma \ref{Sbound} and Proposition \ref{irr just},  $K(v)$ is a compact convex set.  Clearly, $aB\subset K(v)$, so $K(v)$ is a convex body  with the origin in its interior.
	\end{proof}
	
		 Now we establish a uniform estimate for the upper bound of $K(v)$.
	
	\begin{proposition}\label{unif thm}
		Let $G $ be an irreducible closed subgroup of $\On$ and $a > 0$ fixed. For every  $v \in S^{n-1}$, define
		\begin{equation}\label{Kv}
			K(v) = \bigcap_{u \in Gv} \{ x \in \mathbb{R}^n : \langle x, u \rangle \leq a \}.
		\end{equation}
		There exists $R > 0$ (independent of $v$) such that $K(v) \subset B(R)$ for all  $v\in\Sn$.
	\end{proposition}
	
	\begin{proof}
		Let $r > 0$ be the constant from Lemma \ref{int estimate}, so that $B(r)\subset \conv(Gv)$ for any $v\in \Sn$. For any $0\neq x \in K(v)$,  let $w = x / |x|$. Then we have 
		\[
		h_{\conv(Gv)}(w)=\max_{u \in Gv} \langle w, u \rangle \geq r \quad \text{for all $v\in \Sn$}.
		\]
		Hence, there exists $u_{v}\in Gv$ such that $\langle w, u_{v} \rangle \geq r$.
		Since $x \in K(v)$ and $u_v \in Gv$, from the definition of \eqref{Kv}, we have $\langle x, u_v \rangle \leq a$. Thus
		\[
		a\ge	\langle x, u_v \rangle = |x| \langle w, u_v \rangle \geq |x| r ,
		\]
		which implies:
		\[
		|x| \leq \frac{a}{r}.
		\]
		Taking $R = a / r > 0$ gives the uniform bound $K(v) \subset B(R)$ for all  $v\in \Sn$.
	\end{proof}

	 If $G_1$ is irreducible and $G_1 \subset G_2$, then $G_2$ is also irreducible.
	Next we present several classical examples of finite irreducible groups, some of which have appeared in \cite{BKMZ group}.

	\begin{enumerate}
		\item The symmetry group $G_n$ of a regular $n$-simplex centered at the origin in $\mathbb{R}^n$ 
		(regular triangle, regular tetrahedron, etc.) 
		is isomorphic to the symmetric group $S_{n+1}$. 
		Since $-I \notin G_n$,  Proposition \ref{finite -i} implies that  the class of $G_n$-invariant 
		convex bodies is not contained in $\mathcal{K}^n_e$. 
		
		For example, when $n=2$, the symmetry group of an equilateral triangle consists of 6 elements: identity map,
	rotation by $120^\circ$,
		 rotation by $240^\circ$, and the reflections across the three axes that pass through a vertex and are perpendicular to the opposite side.

		\item The rotation symmetry group (i.e., orientation-preserving symmetries) of a regular $n$-simplex centered at the origin in $\mathbb{R}^n$   is isomorphic to the alternating group $A_{n+1}$.
		For example, the rotation symmetry group $G$ of an equilateral triangle consists of rotations by $0^\circ$, $120^\circ$, and $240^\circ$.
		
		The class of $G$-invariant convex bodies is rich. For instance,  we take a suitable convex arc connecting two vertices of an equilateral triangle, and rotate this arc by $120^\circ$ and $240^\circ$ to form the boundary of a convex body. For detailed constructions, see Construction 3.4 in \cite{BKMZ group}.
		

		\item For  $m\ge3$, the dihedral group $D_{m}$ is the symmetry group of a regular $m$-gon in the plane. If $m$ is odd, then $-I\notin D_{m}$. From Proposition \ref{finite -i}, the class of $D_m$-invariant convex bodies is not a subset of $\mathcal{K}^2_e$.
		
		\item The cyclic group $\mathbb{Z}_m$ is a subgroup of $D_m$ for $m \geq 3$, and it constitutes the rotational symmetry group of a regular $m$-gon (generated by the rotation $\frac{2\pi}{m}$).

		\item 	The symmetry group of the cube \([-1, 1]^n\) in \(\mathbb{R}^n\) is the Weyl group $ (\mathbb{Z}/2\mathbb{Z})^n \rtimes S_n$ (order $2^{n}n!$).
			Each element $(\epsilon_1,\ldots,\epsilon_n; \sigma)$ corresponds to an $n \times n$ generalized permutation matrix (having exactly one nonzero entry $\pm 1$ in each row and column), where
			the permutation $\sigma$ determines the positions of the nonzero entries, and
			the signs $\epsilon_i$ determine whether these nonzero entries are $+1$ or $-1$.
			When \(n \geq 3\) is odd, its rotation symmetry group does not contain \(-I\). 

		\item  Since in $\mathbb{R}^{1}$ there exist no nontrivial subspaces (the only subspaces of $\mathbb{R}^1$ are $\{0\}$ and $\mathbb{R}^1$), any subgroup of $\mathrm{O}(1)$ (namely $\{1\}$ or $\{\pm 1\}$) is automatically irreducible.
	\end{enumerate}
	
	For a finite irreducible group $G\subset \On$ with $-I \notin G$,  the proof of Proposition~\ref{finite -i}  shows that for almost every $x \in \mathbb{R}^n$,  $ -x\notin Gx$. For any such $x$, $\mathrm{conv}(Gx)$ is a convex body that is  not origin-symmetric. Consequently, $\mathrm{conv}(Gx) + tB$ forms a family of $G$-invariant convex bodies that are not origin-symmetric for any $t \geq 0$.

	There have been extensive studies on the algebraic properties of irreducible representations. For example,
	\cite{Zimmermann} classified all irreducible subgroups of $\SOn$, and \cite{Fulton repre} investigated  the irreducible representations of finite groups. \cite{real repre} discussed  the classification of real representations of a finite
	group. In particular,  for a finite group $G$ acting irreducibly on $\mathbb{R}^n$, the endomorphism ring $\operatorname{End}_G( \mathbb{R}^n) $ of $G$-equivariant linear maps is isomorphic to one of 
	$\mathbb{R}$, $\mathbb{C}$, or $\mathbb{H}$  (the quaternions).

	\vspace{1\baselineskip} 
	\section{Existence of Solutions to the $L_{p}$ Dual Minkowski Problem for $q\neq 0$}

	To solve the Minkowski problem via variational methods,  we hope to transform the $L_{p}$ dual Minkowski problem into an optimization problem.

		Let $G$ be an irreducible closed subgroup of $\On$ with $n\ge2$ and $Q$ a $G$-invariant star body in $\Rn$.
	For any nonzero finite Borel measure  $\mu$ on $\Sn$,
	we define  $C^{+}(\Sn)$ as the set of continuous functions $f: \Sn \to (0, \infty)$. If $q\neq0$, $p\neq0$, we introduce a functional $\Phi_{p}: C^{+}(\Sn) \to \R$ given by
	
	\begin{equation}\label{def pneq0}
		\Phi_{p}(f) = \frac{1}{p} \log \int_{S^{n-1}} f^{p}  d\mu - \frac{1}{q} \log \widetilde{V}_{q}([f], Q).
	\end{equation}
	
	If $q\neq0$, $p=0$,  define 	
	
	\begin{equation}\label{defp=0}
		\Phi_{0}(f) = \frac{1}{|\mu|} \int _{S^{n-1}}\log f  d\mu - \frac{1}{q} \log \widetilde{V}_{q}([f], Q),
	\end{equation}
	where $ [f] $ is the Wulff shape generated by $ f $:
	\begin{equation}
		[f]=\left\{x\in \mathbb{R}^{n}: \la x, v \ra \le f(v),\; \text{for all}\; v\in S^{n-1}\right\}.
	\end{equation}
	Note that $\Phi_{p}$ is homogeneous of degree $0$, i.e., for all $\lambda >0$ and $f\in C^{+}(\Sn)$, we have 
	\begin{equation}\label{homo}
			\Phi_{p}(\lambda f) = \Phi_{p}(f).
	\end{equation}

	\begin{lemma}\label{inf sup}
		For $q\neq0$, $p\in\R$, and $f\in C^{+}(\Sn)$, we have
		\begin{equation}
				\Phi_{p}(h_{[f]})\le \Phi_{p}(f).
		\end{equation}
	\end{lemma}

	\begin{proof}
		From the definition of the Wulff shape, we have $0<h_{[f]}\le f$ and $[h_{[f]}]=[f]$. Therefore, if $p<0$,
		$$h_{[f]}^{p}\ge f^{p}, \quad \quad \frac{1}{p} \log \int_{S^{n-1}} h_{[f]}^{p}  d\mu \le \frac{1}{p} \log \int_{S^{n-1}} f^{p}  d\mu.$$
		If $p>0$,
		$$h_{[f]}^{p}\le f^{p}, \quad \quad \frac{1}{p} \log \int_{S^{n-1}} h_{[f]}^{p}  d\mu \le \frac{1}{p} \log \int_{S^{n-1}} f^{p}  d\mu.$$
		If $p=0$,
		$$h_{[f]}\le f,\quad \quad \int _{S^{n-1}}\log h_{[f]} d\mu \le \int _{S^{n-1}}\log f d\mu.$$
		From  definitions \eqref{def pneq0} and \eqref{defp=0}, the result is proved.
	\end{proof}
	
	Define the set $\mathcal{L}$ as:
	\begin{equation}\label{L}
		\mathcal{L}=\{f\in C^{+}(\Sn): \text{$f$ is $G$-invariant and $\widetilde{V}_{q}([f], Q)=1$ }\}.
	\end{equation}
	By  definitions \eqref{def pneq0} and \eqref{defp=0}, for $f\in\mathcal{L}$, we have that when $p\neq 0$,
		\begin{equation}\label{Lpneq0}
		\Phi_{p}(f) = \frac{1}{p} \log \int_{S^{n-1}} f^{p}  d\mu,
	\end{equation}
	and when $p=0$, 
	\begin{equation}\label{Lq=0}
		\Phi_{0}(f) = \frac{1}{|\mu|} \int _{S^{n-1}}\log f  d\mu.
	\end{equation}

	Due to the homogeneity of $\Phi_{p}$ in \eqref{homo}, we can restrict the optimization problem for $G$-invariant functions to the set $\mathcal{L}$, where local extrema become global. This is a standard approach for homogeneous Minkowski problems. In contrast, for non-homogeneous cases like the Gaussian Minkowski problem, see, e.g., \cite{Huang2021,Liu2022,shan2025}.

	Note that, if $f\in C^{+}(\Sn)$ is a $G$-invariant function, then the Wulff shape $[f]$ is clearly a $G$-invariant convex body. Since for any $g\in G$,
	\begin{align*}
		g[f]=&\left\{gx\in \mathbb{R}^{n}: \la x, u \ra \le f(u),\;  u\in S^{n-1}\right\}\\
		=&\left\{y\in \mathbb{R}^{n}: \la g^{-1}y, u \ra \le f(u),\;  u\in S^{n-1}\right\}\\
		=&\left\{y\in \mathbb{R}^{n}: \la y, gu \ra \le f(gu),\;  u\in S^{n-1}\right\}\\
				=&\left\{y\in \mathbb{R}^{n}: \la y, v \ra \le f(v),\;  v\in S^{n-1}\right\}=[f].
	\end{align*}
	Conversely, if $K$ is a $G$-invariant convex body, then its support function is naturally $G$-invariant from \eqref{eq:support-transform}.
	
	Next, we consider the minimization problem for the functional $\Phi_{p}$.
	
	\begin{lemma}\label{unif bound}
		Let $q\neq 0$, $p\in\R$. $\{K_i\}$ is a family of $G$-invariant convex bodies. If $h_{K_{i}}\in \mathcal{L}$ such that $\lim_{i \to \infty} R_{i} = +\infty$, where  $R_{i}=\max\{|x|: x\in K_{i}\}$, then
		\begin{equation}
		\lim_{i \to \infty} \Phi_{p}(h_{K_{i}}) = +\infty.
		\end{equation}
	\end{lemma}
	
	\begin{proof}
	From Proposition \ref{unif thm} with $a=1$, there exists $m>0$ such that 
	\begin{equation*}
		K(v) = \bigcap_{u \in Gv} \{ x \in \mathbb{R}^n : \langle x, u \rangle \leq  1 \} \subset mB
	\end{equation*}
	 for every  $v \in S^{n-1}$. Then 
	 \begin{equation}\label{KvB}
	  \bigcap_{u \in Gv} \{ x \in \mathbb{R}^n : \langle x, u \rangle \leq   \frac{1}{2m} \} \subset \frac{1}{2}B
	 \end{equation}
	 for every  $v \in S^{n-1}$. Since $\lim_{i \to \infty} R_{i} = +\infty$, there exists a large $N$ such that for $i>N$, $R_{i}^{-\frac{1}{2}}< \frac{1}{2m}$. We choose $v_{i}\in\Sn$ such that $R_{i}v_{i}\in K_{i}$, by \eqref{KvB},
	 \begin{equation}\label{inclusion}
\bigcap_{u \in Gv_{i}} \{ x \in \mathbb{R}^n : \langle x, u \rangle \leq  R_{i}^{-\frac{1}{2}}  \}\subset	 	 \bigcap_{u \in Gv_{i}} \{ x \in \mathbb{R}^n : \langle x, u \rangle \leq   \frac{1}{2m} \} \subset \frac{1}{2}B.
	 \end{equation} 
	 Therefore, for $i>N$, from \eqref{inclusion}, if $u\in\Sn$, then 
	 \begin{equation*}
	 	u\notin \bigcap_{w \in Gv_{i}} \{ x \in \mathbb{R}^n : \langle x, w \rangle \leq  R_{i}^{-\frac{1}{2}}  \}.
	 \end{equation*}
	 In other words, for any $u\in\Sn$, there exists $w\in G v_{i} $ such that $\la u, w\ra>R_{i}^{-\frac{1}{2}} $. Since $K_{i}$ is $G$-invariant, we have $R_{i}w\in K_{i}$. Hence, 
	 \begin{equation}\label{sup low}
	 	h_{K_{i}}(u)\ge \la u, R_{i}w \ra=R_{i}\la u, w \ra\ge R_{i}^{\frac{1}{2}},\quad \text{$\forall u\in\Sn$}.
	 \end{equation}
	 
	 Therefore, by \eqref{sup low}, if $p> 0$,
	 \begin{equation*}
	 	 \log \int_{S^{n-1}} h_{K_{i}}^{p}  d\mu\ge 	 \log \int_{S^{n-1}} R_{i}^{\frac{p}{2}}  d\mu \rightarrow +\infty,
	 \end{equation*}
	 which implies 
	 \begin{equation}\label{phi1}
	 	\Phi_{p}(h_{K_{i}}) = \frac{1}{p} \log \int_{S^{n-1}} h_{K_{i}}^{p}  d\mu\rightarrow +\infty.
	 \end{equation}
	 
	 If $p<0$, 
	  \begin{equation*}
	 	\log \int_{S^{n-1}} h_{K_{i}}^{p}  d\mu\le 	 \log \int_{S^{n-1}} R_{i}^{-\frac{|p|}{2}}  d\mu \rightarrow -\infty,
	 \end{equation*}
	 which implies 
	 \begin{equation}\label{phi2}
	 	\Phi_{p}(h_{K_{i}}) = \frac{1}{p} \log \int_{S^{n-1}} h_{K_{i}}^{p}  d\mu\rightarrow +\infty.
	 \end{equation}
	 
	 If $p=0$
	 \begin{equation}\label{phi3}
	 	\Phi_{0}(h_{K_{i}}) = \frac{1}{|\mu|} \int _{S^{n-1}}\log h_{K_{i}}  d\mu \ge \frac{1}{|\mu|} \int_{S^{n-1}} \log R_{i}^{\frac{1}{2}}  d\mu \rightarrow +\infty.
	 \end{equation}
	 Combining \eqref{phi1}, \eqref{phi2} and \eqref{phi3}, we complete the proof.
	\end{proof} 

	We require the following estimate concerning degeneration.
	
	\begin{lemma}\label{degener}
		 Let \(K\) be a \(G\)-invariant nonempty compact convex set in \(\mathbb{R}^n\). If the support function \(h_K(u_{0}) = 0\) for some \(u_{0} \in S^{n-1}\), then \(K = \{0\}\).
		
	\end{lemma}
	
	\begin{proof}
	 The condition \(h_K(u_{0}) = 0\) implies that \(\langle x, u_{0} \rangle \leq 0\) for all \(x \in K\).
		Since 
		\(K\) is \(G\)-invariant, by \eqref{eq:support-transform}, we have for any \(g \in G\):
		\[
		h_K(gu_{0}) =  h_{g^{-1}K}(u_{0})
		= h_K(u_{0}) = 0.
		\]
	 Thus, \(\langle x, gu_{0} \rangle \leq 0\) for all \(x \in K\) and all \(g \in G\).

		Now define the set:
		\[
		L = 
		\{ x \in \mathbb{R}^n: \langle x, gu_{0} \rangle \leq 0 \text{ for all } g \in G \}.
				\]
		By the above, 
		\(K \subset L\). 
		Suppose there exists a nonzero vector \( y \in L \). Then \( \lambda y \in L \) for any \( \lambda > 0 \), and hence we may assume \( y \in \Sn \). The condition \( y \in L \) implies that for every \( g \in G \),  \( \langle y, gu_{0} \rangle \leq 0 \); in other words, the orbit \( Gu_{0} \) lies in the closed hemisphere in the direction of \( y \). But $G$ is irreducible. This contradicts Proposition~\ref{irr just}. Therefore, \(L = \{0\}\).

		Since 
		\(K \subset L = \{0\}\) and \(K\) is nonempty, we conclude that \(K = \{0\}\).
	\end{proof}
	

Since $G$ is irreducible and $n\ge 2$, it has no non-zero fixed points.  If there exists a non-zero fixed point $x_1 \in \mathbb{R}^n$, then the line $\mathbb{R}x_1$ forms a nontrivial $G$-invariant subspace, contradicting the irreducibility.  By Proposition \ref{int1}, $\mathcal{K}_{G}\subset \mathcal{K}^{n}_{o}$.

Therefore, we can obtain the existence of solutions to the minimization problem for $\Phi_{p}$.

\begin{proposition}\label{min exsitence}
Let $q\neq0$, $p\in\R$.	There exists a $G$-invariant convex body  $K$  such that 
	\begin{equation}\label{minget}
			\Phi_{p}(h_{K})=\min \{\Phi_{p}(f): f\in \mathcal{L}\}.
	\end{equation}
\end{proposition}

\begin{proof}
	From Lemma \ref{inf sup}, we can assume that $\{h_{K_{i}}\}\subset\mathcal{L}$ is a minimizing sequence such that 
	$$\lim_{i \to \infty}\Phi_{p}(h_{K_{i}})=\inf \{\Phi_{p}(f): f\in \mathcal{L}\}.$$
	By Lemma \ref{unif bound}, if  $\lim_{i \to \infty} R_{i} = +\infty$ (where  $R_{i}=\max\{|x|: x\in K_{i}\}$), then
	$\lim_{i \to \infty} \Phi_{p}(h_{K_{i}}) = +\infty.$ This contradicts the  fact that $\{\Phi_{p}(h_{K_{i}})\}$ converges to the infimum.  Therefore, there exists a constant $R>0$ such that
	\begin{equation} \label{ll}
	K_{i}\subset RB
	\end{equation}
	for all $i$.  Combining \eqref{ll} with the   Blaschke selection theorem, we can find a subsequence, still denoted by $\{K_{i}\}$, such that $K_{i}\rightarrow K$ for some compact convex set $K$. For any $g\in G$, since $G \subset \On$ and  each $K_{i}$ is $G$-invariant, we have $d_{H}(K_{i}, gK)=d_{H}(gK_{i}, gK)=d_{H}(K_{i}, K)\rightarrow 0$. Therefore, $gK = K$, and thus $K$ is also $G$-invariant.  The $G$-invariant convex bodies $K_i$ satisfy $0 \in \operatorname{int} K_i$. Since $K_i \to K$, it follows that $0 \in K$. We now prove that $0 \in \operatorname{int} K$. If $0 \in \partial K$, then there exists $u \in S^{n-1}$ such that $h_K(u) = 0$. From Lemma \ref{degener}, $K = \{0\}$. Therefore, for any small $\epsilon > 0$, there always exists $K_i \subset \epsilon B$ for sufficiently large $i$, which contradicts $V_q(K_i, Q) = 1$. Thus, $K$ is a  $G$-invariant convex body.
	
	 By the continuity of the dual mixed volume in Lemma \ref{continuity}, we have $h_{K}\in\mathcal{L}$. Hence, $$\lim_{i \to \infty}\Phi_{p}(h_{K_{i}})=\Phi_{p}(h_{K})=\min \{\Phi_{p}(f): f\in \mathcal{L}\}.$$ 
\end{proof}

The $G$-invariance of the $(p, q)$-th dual curvature measure will be utilized:
\begin{lemma}\label{measure G invariant}
	Let $K$ be a $G$-invariant convex body, and $Q$ a $G$-invariant star body. Then the measure $\widetilde{C}_{p,q}(K,Q, \cdot)$ is $G$-invariant.
\end{lemma}

\begin{proof}
For any Borel set $\eta\subset \Sn$,	let $f(v)=1_{\eta}(v)$ in \eqref{measure trans}. For any $g \in G$, since $K$ and $Q$ are $G$-invariant, we have $g^{-1} K = K$ and $g^{-1} Q = Q$.
 Then,
	\begin{align*}
\widetilde{C}_{p,q}(K,Q, \eta)&=	\int_{S^{n-1}} 1_{\eta}(v) \, d\widetilde{C}_{p,q}(g^{-1} K, g^{-1} Q, v)\\
& = \int_{S^{n-1}} 1_{\eta} \left( g^{-1}v \right) \, d\widetilde{C}_{p,q}(K, Q, v)\\
&=\widetilde{C}_{p,q}(K,Q, g\eta).
\end{align*}
\end{proof}

We need  the following lemma  about group invariant measures from \cite{BKMZ group}:

\begin{lemma}[Lemma 5.1, \cite{BKMZ group}]\label{unique}
	Let $G$ be a closed subgroup of $\On$, $n \ge 2$. If $\mu_1$ and $\mu_2$ are $G$-invariant finite Borel measures on $S^{n-1}$, then $\mu_1 = \mu_2$ if and only if $\int_{S^{n-1}} g d\mu_1 = \int_{S^{n-1}} g d\mu_2$ for any $G$-invariant continuous function $g : S^{n-1} \to \mathbb{R}$.
\end{lemma}

We now demonstrate that the minimizer $K$ of the optimization problem yields a solution (up to a constant factor) to the corresponding $L_p$ dual Minkowski problem.

\begin{theorem}\label{closed minkowski}
	Let $q \neq 0$, $p \in \mathbb{R}$, $G$ be an  irreducible closed subgroup of $\mathrm{O}(n)$ with $n \geq 2$, and $Q$ be a $G$-invariant star body in $\mathbb{R}^n$. For any non-zero finite Borel measure $\mu$ on $S^{n-1}$, $\mu$ is $G$-invariant if and only if there exists a $G$-invariant convex body $K$ in $\mathbb{R}^n$ such that
	 $$\mu=\widetilde{C}_{p,q}(K, Q,\cdot),\quad \text{$p\neq q$,}\quad \text{and} \quad \mu=\left( \int_{S^{n-1}}h_{K}^{p}d\mu\right) \cdot \widetilde{C}_{p,q}(K, Q, \cdot),\quad\text{$p=q$}.$$
\end{theorem}

\begin{proof}
	Let $h_{K}$ be the minimizer of functional $\Phi_{p}$ in Proposition \ref{min exsitence} with $\widetilde{V}_{q}(K,Q)=1$.
By the homogeneity of $\Phi_{p}$ in \eqref{homo}, $h_K$ is the global minimizer of $\Phi_{p}$.	For any continuous $G$-invariant function $f:\Sn\rightarrow \mathbb{R}$,   the function $\left( h_{K}+tf\right) $ belongs to $C^{+}(\Sn)$ and is $G$-invariant for sufficiently small $|t|$.  Let $g(t)=\Phi_{p}(h_{K}+tf)$ for sufficiently small $|t|$.
	From \eqref{minget}, we obtain
	\begin{equation}\label{ft}
		g(t)=\Phi_{p}(h_{K}+tf)\ge \Phi_{p}(h_{K})=g(0).
	\end{equation}
	
	If $p\neq 0$, by \eqref{def pneq0},
	\begin{equation*}
		g(t) = \frac{1}{p} \log \int_{S^{n-1}} (h_{K}+tf)^{p}  d\mu - \frac{1}{q} \log \widetilde{V}_{q}([h_{K}+tf], Q).
	\end{equation*}
	Therefore, by \eqref{ft} and \eqref{dual vari}, if $p\neq 0$,  
	
	\begin{equation*}
			0=\left.\frac{d}{d t}\right|_{t=0}	g(t)=\dfrac{\int_{S^{n-1}}fh_{K}^{p-1}d\mu}{\int_{S^{n-1}}h_{K}^{p}d\mu}-\int_{S^{n-1}}\frac{f}{h_{K}}d\widetilde{C}_{q}(K,Q).
	\end{equation*}
	In other words,
	\begin{equation}\label{=}
		\int_{S^{n-1}}fh_{K}^{p-1}d\mu=\lambda \int_{S^{n-1}}\frac{f}{h_{K}}d\widetilde{C}_{q}(K,Q)
	\end{equation}
	for $\lambda=\int_{S^{n-1}}h_{K}^{p}d\mu$. Therefore, for any continuous $G$-invariant function $\phi$, since $K$ is $G$-invariant, the function $f=\frac{\phi}{h_{K}^{p-1}}$ is also $G$-invariant. Substituting into \eqref{=} and applying \eqref{pqdef}, we have
	\begin{equation}\label{subs}
		\int_{S^{n-1}}\phi d\mu=\lambda \int_{S^{n-1}}\phi h_{K}^{-p}d\widetilde{C}_{q}(K,Q)=\lambda\int_{S^{n-1}}\phi d\widetilde{C}_{p,q}(K,Q).
	\end{equation}
	Then by \eqref{subs}, Lemma \ref{measure G invariant}, and Lemma \ref{unique}, we obtain $\mu=\lambda \widetilde{C}_{p,q}(K,Q,\cdot)$.	
When $p\neq q$, we use the $(q-p)$-th homogeneity of $\widetilde{C}_{p,q}(K, Q, \cdot)$ in $K$. Defining $K' = \lambda^{\frac{1}{q-p}} K$, we obtain 
	$\mu= \widetilde{C}_{p,q}(K',Q,\cdot)$.	
	
	Similarly, if $p=0$, 
	\begin{equation*}
		g(t) = \frac{1}{|\mu|} \int _{S^{n-1}}\log (h_{K}+tf) d\mu - \frac{1}{q} \log \widetilde{V}_{q}([h_{K}+tf], Q),
	\end{equation*}
then	 we have
	
	\begin{equation*}
		0=\left.\frac{d}{d t}\right|_{t=0}	g(t)=\frac{1}{|\mu|}\int_{S^{n-1}}\frac{f}{h_{K}}d\mu-\int_{S^{n-1}}\frac{f}{h_{K}}d\widetilde{C}_{q}(K,Q).
	\end{equation*}
	In other words,
	\begin{equation}
		\int_{S^{n-1}}\phi d\mu=|\mu| \int_{S^{n-1}}\phi d\widetilde{C}_{q}(K,Q)
	\end{equation}
	for any  continuous $G$-invariant function $\phi$. From Lemma \ref{measure G invariant} and Lemma \ref{unique}, we obtain $\mu=|\mu|\widetilde{C}_{q}(K,Q,\cdot)=\widetilde{C}_{q}(|\mu|^{\frac{1}{q}}K,Q,\cdot)$.
\end{proof}

\vspace{1\baselineskip} 

\section{The $L_{p}$ Aleksandrov Problem}
The proof of the $L_p$ dual Minkowski problem for nonzero $q$ relies on the geometric characterization of irreducible groups, which reveals a profound relationship between irreducibility and Minkowski-type problems. In this section, we adopt an alternative perspective by analyzing the geometric symmetries of convex bodies invariant under irreducible group actions, which allows us to establish the existence of solutions to the $L_p$ Aleksandrov problem.

Let $G$ be an irreducible closed subgroup of $\On$ with $n\ge2$ and $Q$ a $G$-invariant star body in $\Rn$.
For any nonzero finite Borel measure $\mu$ on $\Sn$, if  $p \neq 0$, we introduce a functional $\mathcal{M}_{p}: C^{+}(\Sn) \to \R$ given by

\begin{equation}\label{Mp}
	\mathcal{M}_{p}(f) = \frac{1}{p} \log \int_{S^{n-1}} f^{p}  d\mu - \frac{1}{V(Q)}  \widetilde{E}([f], Q).
\end{equation}

If  $p=0$,  define 	

\begin{equation}\label{M0}
	\mathcal{M}_{0}(f) = \frac{1}{|\mu|} \int _{S^{n-1}}\log f  d\mu - \frac{1}{V(Q)}  \widetilde{E}([f], Q).
\end{equation}
Here $\widetilde{E}$ denotes the dual mixed entropy in  \eqref{entropy def} and $V(Q)$ represents the volume of $Q$.
Note that $\mathcal{M}_{p}$ is homogeneous of degree $0$, i.e., for all $\lambda >0$ and $f\in C^{+}(\Sn)$, we have 
\begin{equation}\label{homoM}
	\mathcal{M}_{p}(\lambda f) = \mathcal{M}_{p}(f)
\end{equation}
and
	\begin{equation}\label{supM}
		\mathcal{M}_{p}(h_{[f]})\le \mathcal{M}_{p}(f)
	\end{equation}
for any $p\in \mathbb{R}$.

	We next present a  result characterizing the geometry of irreducible groups. This helps us understand the symmetry inherent in the corresponding group-invariant convex bodies.

Ellipsoids are fundamental objects in convex geometry.
The
celebrated \textit{John ellipsoid}, which was introduced by John,
is an extremely useful tool in convex geometry and Banach space
geometry (see, e.g., \cite{LYZ2000,LYZ2005}). Associated with a convex body $ K $ in $ \mathbb{R}^{n} $,
the \textit{John ellipsoid} of $ K $ is the unique ellipsoid of
maximal volume contained in $ K $.  For $G$-invariant convex bodies corresponding to the irreducible group $G$, we have:

\begin{proposition}\label{john ellipsoid}
	Let $G $ be an irreducible subgroup of $\On$ with $n\ge 2$. Let $K \subset \mathbb{R}^n$ be a $G$-invariant convex body.  Then the John ellipsoid of $K$  is a ball centered at the origin.
\end{proposition}

\begin{proof}
	Let $E$ be the John ellipsoid of $K$.
	Since $G \subset \On$ preserves volume and $K$ is $G$-invariant, 	for any $g \in G $,  $gE$ is  an ellipsoid contained in $gK = K$ with $V(gE) = V(E)$. By the uniqueness of the John ellipsoid, we  have $gE = E$. Therefore $E$ is $G$-invariant. Its centroid is a fixed point of $G$ (see the proof of Proposition \ref{int1}), and since $G$ is irreducible and thus has no nonzero fixed points, the centroid must be the origin.
	Consequently, $E$ can be expressed as:  
	\begin{equation}\label{ellipsoid}
		E = \{ x \in \Rn \mid x^\top A x \leq 1 \}
	\end{equation}
	for some symmetric positive definite matrix $A \in \mathrm{GL}(n, \R)$. For any $g\in G$,  since $g^{-1} = g^\top$, we have
	\[
	gE = \{ gx \mid x^\top A x \leq 1 \} = \{ y \mid (g^{-1}y)^\top A (g^{-1}y) \leq 1 \} = \{ y \mid y^\top (gAg^{-1}) y \leq 1 \}.
	\]
	Equating $gE = E$ gives $gAg^{-1} = A$. 
	In other words, 
	\begin{equation}\label{commute}
		Ag = gA  \quad \text{for all} \quad g \in G.
	\end{equation}
	Thus, $A$ commutes with every element of $G$.

	Since $A$ is a symmetric positive definite matrix, its eigenvalues are real numbers. Let $\lambda$ be an eigenvalue of $A$, and $V_\lambda \subset \Rn$ its corresponding eigenspace:
	\[
	V_\lambda = \{ v \in \Rn \mid Av = \lambda v \}.
	\]
	For any $v \in V_\lambda$ and any $g \in G$, by \eqref{commute}, we have
	\[
	A(gv) = g(Av) = g(\lambda v) = \lambda (gv).
	\]
	This shows $gv \in V_\lambda$, so $V_\lambda$ is $G$-invariant.

	By the irreducibility of $G$, the only $G$-invariant subspaces of $\Rn$ are $\{0\}$ and $\Rn$. Since $V_\lambda \neq \{0\}$ (as $\lambda$ is an eigenvalue), we must have $V_\lambda = \Rn$. Thus, every vector in $\Rn$ is an eigenvector of $A$ with eigenvalue $\lambda$, implying $A = \lambda I$ for some $\lambda > 0$ (since $A$ is positive definite).

	Substituting $A = \lambda I$ into the equation  \eqref{ellipsoid}
	for $E$  yields:
	\[
	E = \{ x \in \Rn \mid x^\top (\lambda I) x \leq 1 \} = \{ x \in \Rn \mid \lambda |x|^2 \leq 1 \} = \left\{ x \in \Rn \bigm| |x| \leq \lambda^{-1/2} \right\},
	\]
	which is a ball centered at the origin with radius $\lambda^{-1/2}$. Therefore, the John ellipsoid of $K$ is a ball centered at the origin.
\end{proof}

This demonstrates the symmetry of group-invariant convex bodies for irreducible groups. Therefore, we have:

\begin{lemma}\label{uniM}
Let \(K\) be a \(G\)-invariant convex body in \(\mathbb{R}^n\) with  \(V(K) = 1\). Then there exist constants \(m, M > 0\) depending only on \(n\)
	such that
	\[
	mB 
	\subset K \subset
	MB.
	\]

\end{lemma}

\begin{proof}
	Since 
	\(G\) is irreducible and \(K\) is \(G\)-invariant, from Proposition \ref{john ellipsoid}, the John ellipsoid \(E\) of \(K\)  must be a ball centered at the origin. Let \(E = rB\) for some \(r > 0\). Then,
as is well known (see Schneider \cite{schneiderbook2014}, p.~588),
	\begin{equation}\label{john inclusion}
	rB 
	\subset K \subset
	nrB.
	\end{equation}
	
	 Since 
	\(V(K) = 1\), we have:
	\[
r^{n}V(B)=V(rB)\leq	1 = V(K) 
	\leq
	V(nrB) = (nr)^n V(B),
	\]
	which implies:
	\begin{equation}\label{raduis}
\frac{1}{n V(B)^{1/n}}\leq	r  	\leq \frac{1}{V(B)^{1/n}}.
	\end{equation}
	Combining \eqref{john inclusion} and \eqref{raduis}, we get:
	\[
	\frac{1}{n V(B)^{1/n}} B \subset rB \subset K \subset nrB \subset \frac
	{n}{V(B)^{1/n}} B.
	\]
	Taking 
	\(m = \frac{1}{n V(B)^{1/n}}\) and \(M = \frac{n}{V(B)^{1/n}}\), we obtain the desired uniform bounds:
	\[
	mB 
	\subset K \subset
	MB.
	\]
	
\end{proof}

Define the set $\mathcal{K}$ as:
\begin{equation}\label{L}
	\mathcal{K}=\{f\in C^{+}(\Sn): \text{$f$ is $G$-invariant and $V([f])=1$ }\}.
\end{equation}

We can then obtain the existence of solutions to the minimization problem for $\mathcal{M}_{p}$.

\begin{proposition}\label{min Mp}
	Let  $p\in\R$.	There exists a $G$-invariant convex body  $K$  such that 
	\begin{equation}
		\mathcal{M}_{p}(h_{K})=\min \{\mathcal{M}_{p}(f): f\in \mathcal{K}\}.
	\end{equation}
\end{proposition}

\begin{proof}
By \eqref{supM}, we can assume that $\{h_{K_{i}}\}\subset\mathcal{K}$ is a minimizing sequence such that 
	$$\lim_{i \to \infty}\mathcal{M}_{p}(h_{K_{i}})=\inf \{\mathcal{M}_{p}(f): f\in \mathcal{K}\}.$$
	From Lemma \ref{uniM},  there exist  constants $m, M>0$ such that
	\begin{equation} \label{subset}
	mB\subset	K_{i}\subset MB
	\end{equation}
	for all $i$.  Combining \eqref{subset} with the   Blaschke selection theorem, we can find a subsequence, still denoted by $\{K_{i}\}$, such that $K_{i}\rightarrow K$ for some  convex body $K$. Similar to Proposition \ref{min exsitence}, $K$ is $G$-invariant and $V(K) = 1$.
	Therefore, we have $h_{K}\in\mathcal{K}$. Hence, $$\lim_{i \to \infty}\mathcal{M}_{p}(h_{K_{i}})=\mathcal{M}_{p}(h_{K})=\min \{\mathcal{M}_{p}(f): f\in \mathcal{K}\}.$$ 
\end{proof}

We now establish the existence of solutions to the $L_p$ Aleksandrov problem.

\begin{theorem}\label{closed Aleksandrov}
	Let  $G$ be an  irreducible closed subgroup of $\mathrm{O}(n)$ with $n \geq 2$, and $Q$ be a $G$-invariant star body in $\mathbb{R}^n$. Let $\mu$ be a non-zero finite Borel measure  on $S^{n-1}$.
	\begin{enumerate}
		\item\label{item1} If $p\neq 0$, then	$\mu$ is $G$-invariant if and only if there exists a $G$-invariant convex body $K$ in $\mathbb{R}^n$ such that
		$\mu=\widetilde{C}_{p,0}(K, Q,\cdot)$.
		\item\label{item2} $\mu$ is $G$-invariant and $| \mu|=V(Q)$ if and only if there exists a $G$-invariant convex body $K$ in $\mathbb{R}^n$ such that
		$\mu=\widetilde{C}_{0}(K, Q,\cdot)$.
	\end{enumerate}

\end{theorem}

\begin{proof}
	Let $h_{K}$ be the minimizer of functional $\mathcal{M}_{p}$ in Proposition \ref{min Mp} with $V(K)=1$.
	By the homogeneity of $\mathcal{M}_{p}$ in \eqref{homoM}, $h_K$ is the global minimizer of $\mathcal{M}_{p}$.	For any continuous $G$-invariant function $f:\Sn\rightarrow \mathbb{R}$,   the function $\left( h_{K}+tf\right) $ belongs to $C^{+}(\Sn)$ and is $G$-invariant for sufficiently small $|t|$.  Let $s(t)=\mathcal{M}_{p}(h_{K}+tf)$ for sufficiently small $|t|$. Then we have
	\begin{equation}\label{Mt}
		s(t)=\mathcal{M}_{p}(h_{K}+tf)\ge \mathcal{M}_{p}(h_{K})=s(0).
	\end{equation}
	
	\textbf{Proof of \eqref{item1}}:
	If $p\neq 0$, 
	\begin{equation*}
		s(t) = \frac{1}{p} \log \int_{S^{n-1}} (h_{K}+tf)^{p}  d\mu - \frac{1}{V(Q)}  \widetilde{E}([h_{K}+tf], Q).
	\end{equation*}
	Therefore, by \eqref{Mt} and \eqref{entropy vari}, we have
	
	\begin{equation*}
		0=\left.\frac{d}{d t}\right|_{t=0}	s(t)=\dfrac{\int_{S^{n-1}}fh_{K}^{p-1}d\mu}{\int_{S^{n-1}}h_{K}^{p}d\mu}-\frac{1}{V(Q)}\int_{S^{n-1}}\frac{f}{h_{K}}d\widetilde{C}_{0}(K,Q).
	\end{equation*}
	In other words,
	\begin{equation}\label{==}
		\int_{S^{n-1}}fh_{K}^{p-1}d\mu=\lambda \int_{S^{n-1}}\frac{f}{h_{K}}d\widetilde{C}_{0}(K,Q)
	\end{equation}
	for $\lambda=\frac{\int_{S^{n-1}}h_{K}^{p}d\mu}{V(Q)}$. Therefore, for any continuous $G$-invariant function $\phi$, since $K$ is $G$-invariant, the function $f=\frac{\phi}{h_{K}^{p-1}}$ is also $G$-invariant. Substituting into \eqref{==}, we have
	\begin{equation}\label{subsM}
		\int_{S^{n-1}}\phi d\mu=\lambda \int_{S^{n-1}}\phi h_{K}^{-p}d\widetilde{C}_{0}(K,Q)=\lambda\int_{S^{n-1}}\phi d\widetilde{C}_{p,0}(K,Q).
	\end{equation}
	Then by \eqref{subsM}, Lemma \ref{measure G invariant}, and Lemma \ref{unique}, we obtain $\mu=\lambda \widetilde{C}_{p,0}(K,Q,\cdot)$.	
	 Defining $K' = \lambda^{-\frac{1}{p}} K$, we obtain 
	$\mu= \widetilde{C}_{p,0}(K',Q,\cdot)$.	
	
	\textbf{Proof of \eqref{item2}}: Similarly, if $p=0$, 
	\begin{equation*}
		s(t) = \frac{1}{|\mu|} \int _{S^{n-1}}\log (h_{K}+tf) d\mu - \frac{1}{V(Q)}  \widetilde{E}([h_{K}+tf], Q),
	\end{equation*}
	then	 we have
	
	\begin{equation*}
		0=\left.\frac{d}{d t}\right|_{t=0}	s(t)=\frac{1}{|\mu|}\int_{S^{n-1}}\frac{f}{h_{K}}d\mu-\frac{1}{V(Q)}\int_{S^{n-1}}\frac{f}{h_{K}}d\widetilde{C}_{0}(K,Q).
	\end{equation*}
	In other words,
	\begin{equation}
		\int_{S^{n-1}}\phi d\mu=\frac{|\mu|}{V(Q)} \int_{S^{n-1}}\phi d\widetilde{C}_{0}(K,Q)
	\end{equation}
	for any  continuous $G$-invariant function $\phi$. From Lemma \ref{measure G invariant} and Lemma \ref{unique}, we obtain $\mu=\frac{|\mu|}{V(Q)}\widetilde{C}_{0}(K,Q,\cdot)$. Therefore, if $|\mu|=V(Q)$, it follows naturally that $\mu=\widetilde{C}_{0}(K,Q,\cdot)$.
	
	Conversely, for any $K \in \mathcal{K}_{o}^{n}$, $Q\in \mathcal{S}_{o}^{n}$,   by the definition \eqref{defdual}, we have
	\begin{equation*}
		\widetilde{C}_{0}(K,Q,\Sn)=\frac{1}{n}\int_{\Sn}\rho_{Q}^{n}(u)du=V(Q).
	\end{equation*}
Therefore,	if there exists a $G$-invariant convex body $K$ in $\mathbb{R}^n$ such that
	$\mu=\widetilde{C}_{0}(K, Q,\cdot)$, it must hold that $|\mu|=V(Q)$.
\end{proof}

\section{Extension to Non-Closed Subgroups}\label{open}
In the previous sections, certain results required the subgroup $G$ to be closed. This section extends the previous results to general  subgroups. We begin with the following  lemma:
	 \begin{lemma}\label{ap1}
	 		Let $ G $ be a  subgroup of   $ \On $ and $x\in \Rn$, then $\overline{Gx}=\overline{G}x$, where $\clG$ denotes the closure of $G$ in $\On$.
	 \end{lemma}
	 
	 \begin{proof}
	 	If $y\in \overline{Gx}$, there exists a sequence $\{g_{i}x\}$ with $g_{i}\in G$ such that $g_{i}x\rightarrow y$. Since $\On$ is compact, there exists a convergent subsequence of $\{g_{i}\}$, which we still denote by $\{g_{i}\}$, such that $g_{i}\rightarrow g \in \overline{G}$. Then $g_{i}x\rightarrow gx\in\overline{G}x$. As $g_{i}x\rightarrow y$, it follows that $y=gx\in\overline{G}x$.
	 	
	 	Conversely, if $z\in\overline{G}x$, $z=gx\in\overline{G}x$ for some $g\in\overline{G}$. Then there exist $g_{i}\in G$ such that $g_{i}\rightarrow g$. Since $g_{i}x \in Gx$ and $g_{i}x\rightarrow gx$, we have $z=gx\in\overline{Gx}$.
	 \end{proof}

	\begin{lemma}\label{ap2}
			Let $ G $ be a  subgroup of   $ \On $. The set of $G$-invariant convex bodies equals 
			the set of $\overline{G}$-invariant convex bodies, i.e., $\mathcal{K}_{G}=\mathcal{K}_{\overline{G}}$.
	\end{lemma}

\begin{proof}
	Since $G\subset \overline{G}$, $\overline{G}$-invariant convex bodies are automatically  $G$-invariant.

	Conversely, let $K$ be a $G$-invariant convex body. For any $g\in\overline{G}$,  there exists a sequence $\{g_{i}\}$ in $G$ such that $g_{i}\rightarrow g$. Since $K$ is $G$-invariant, for any $x\in K$, we have $g_{i}x\in K$. Since $K$ is compact, $g_i x \to gx\in K$. Thus $K$ is  $\overline{G}$-invariant.
\end{proof}

Similarly, if the star body $Q$ is $G$-invariant, then $Q$ is also $\overline{G}$-invariant.

	When the closedness restriction is removed, Propositions \ref{G=ball} and \ref{Gsub osym} no longer hold, as the set $K$ constructed in their proofs may not  be a convex body.  We now present a counterexample:
	
\begin{example}
		 Let $n = 2$ and $G$ be the subgroup of $\mathrm{SO}(2)$ generated by an irrational rotation. Specifically, fix an irrational rotation $\alpha$ ($\frac{\alpha}{\pi}\notin \mathbb{Q}$) and define $G = \{ R_{k\alpha} \mid k \in \mathbb{Z} \}$, where $R_\theta$ denotes the rotation matrix by angle $\theta$. Then $G$ is dense in $\mathrm{SO}(2)$.
		
		 Take $x = (1,0)$. The orbit $Gx = \{ (\cos(k\alpha), \sin(k\alpha)) \mid k \in \mathbb{Z} \}$ is dense in the unit circle, but $-x = (-1,0) \notin Gx$ (since $-x$ corresponds to angle $\pi$, and $k\alpha \equiv \pi \pmod{2\pi}$ has no integer solution; otherwise $\alpha$ would be a rational multiple of $\pi$).  Therefore, the action of $G$ on the circle is not transitive.
		
	  Since $G$ is dense in $\mathrm{SO}(2)$,  Lemma \ref{ap2} implies that a convex body is $G$-invariant if and only if it is $\mathrm{SO}(2)$-invariant.  The only $\mathrm{SO}(2)$-invariant convex bodies are  disks centered at the origin. This contradicts both Proposition \ref{G=ball} and Proposition \ref{Gsub osym}.
	\end{example}

	By making slight modifications to Propositions \ref{G=ball} and \ref{Gsub osym}, we have:
	
		\begin{proposition}\label{open G ball}
		Let \( G \) be a  subgroup of  \( \On \). The following conditions are equivalent:
		\begin{enumerate}
			\item $\KG=\mathcal{B}^{n}$.
			\item For any $v \in \Sn$, the orbit $Gv$ is dense in $\Sn$.
		\end{enumerate}
	\end{proposition}
	
	\begin{proof}
		\medskip\noindent
		\textbf{(2) \(\Rightarrow\) (1):} 
 For any $v\in\Sn$, since $\overline{Gv}=\Sn$,  Lemma \ref{ap1} implies that  $\overline{G}v=\overline{Gv}=\Sn$. Thus, $\overline{G}$  acts transitively on $\Sn$. By Proposition \ref{G=ball}, this implies that the only $\overline{G}$-invariant convex bodies are balls.   From Lemma \ref{ap2}, we conclude that $\KG = \mathcal{K}_{\overline{G}} = \mathcal{B}^{n}$.
		
		\medskip\noindent
		\textbf{(1) \(\Rightarrow\) (2):}  
	By Lemma \ref{ap2}, condition $(1)$  implies that every $\overline{G}$-invariant 
	convex body in \( \mathbb{R}^n \) is a Euclidean ball centered at the origin.  By Proposition \ref{G=ball},  for any $v\in\Sn$ we have $\overline{G}v=\Sn$. Then, by Lemma \ref{ap1}, we have $\overline{G}v = \overline{Gv}$, so $\overline{Gv}=\Sn$. In other words, the orbit $Gv$ is dense in $\Sn$.
		\end{proof}
		
		Similarly, we have:
		
			\begin{proposition} 
			Let \(G\) be a  subgroup of \(\On\). The following conditions are equivalent:
			\begin{enumerate}
				\item $\KG \subset \mathcal{K}^{n}_{e}$.
				\item For every  \(x \in \Rn \),  \(-x \in \overline{Gx}\).
			\end{enumerate}
		\end{proposition}
		
		Moreover, for any finite Borel measure, $G$-invariance implies $\overline{G}$-invariance:

			\begin{lemma}\label{mu bar}
				Let $G$ be a subgroup of  $\On$, and $\mu$ be a finite Borel measure on the unit sphere $\Sn$. If $\mu$ is $G$-invariant, then $\mu$ is $\clG$-invariant.
			\end{lemma}
			
			\begin{proof}
				We need to show that $\mu(hE) = \mu(E)$ for all $h \in \clG$ and all Borel sets $E \subset \Sn$.

				Since $\On$ is a compact Lie group under the Euclidean topology of $\R^{n^2}$, 
				$\clG$ is  a closed subgroup. The $G$-invariance of $\mu$ means that $g_*\mu = \mu$ 
				for all $g \in G$, where the pushforward measure $g_*\mu$ is defined by $g_*\mu(E) = \mu(g^{-1}E)$.

			Fix $h \in \clG$.		Let $\{g_k\} \subset G$ be a sequence converging to $h \in \clG$. For any  continuous function $f: \Sn \to \R$, we have:
				\[
				\int_{\Sn} f(x)  dg_{k*}\mu(x) = \int_{\Sn} f(g_k x)  d\mu(x), 
				\]
				and
				\[
				\int_{\Sn} f(x)  dh_*\mu(x) = \int_{\Sn} f(h x)  d\mu(x).
				\]
				Since $g_k \to h$ and the action is continuous, $g_k x \to h x$ for each $x \in \Sn$. By continuity of $f$, we have $f(g_k x) \to f(h x)$ pointwise. As $f$ is bounded  and $\mu$ is finite, the Dominated Convergence Theorem implies:
				\[
				\lim_{k \to \infty} \int_{\Sn} f(g_k x)  d\mu(x) = \int_{\Sn} f(h x)  d\mu(x).
				\]
				Thus $g_{k*}\mu \to h_*\mu$ weakly.

				 By $G$-invariance, $g_{k*}\mu = \mu$ for all $k$. The constant sequence $\{\mu\}$ converges weakly to $\mu$, so by uniqueness of weak limits:
				\[
				h_*\mu = \mu.
				\]
				Hence, $\mu(h^{-1}E) = h_*\mu(E) = \mu(E)$ for all Borel sets $E \subset \Sn$, which shows that $\mu$ is $\clG$-invariant.
			\end{proof}
			
			Therefore, for general subgroups (without the closedness restriction) of $\On$, we can still obtain solutions to the corresponding group-invariant $L_{p}$ dual Minkowski problem.
			
			\begin{theorem}\label{general exist}
				Let $q\in\mathbb{R}$, $p \in \mathbb{R}$, $G$ be an irreducible subgroup of $\mathrm{O}(n)$ with $n \geq 2$, and $Q$ be a $G$-invariant star body in $\mathbb{R}^n$. For any non-zero finite Borel measure $\mu$ on $S^{n-1}$, $\mu$ is $G$-invariant if and only if there exists a $G$-invariant convex body $K$ in $\mathbb{R}^n$ such that
			\[
			\mu = \widetilde{C}_{p,q}(K, Q, \cdot) \quad \text{when  $p \neq q$},
			\]
			and
			\[
			\mu = \lambda \widetilde{C}_{p,q}(K, Q, \cdot) \quad \text{for some $\lambda>0$ when $p=q$}.
			\]
			\end{theorem}
			
			\begin{proof}
				If $\mu$ is $G$-invariant, by Lemma \ref{mu bar},  $\mu$ is $\overline{G}$-invariant. Since $G$ is irreducible and $G\subset \overline{G}$, the closure $\overline{G}$ is also irreducible. Moreover, since $Q$ is $G$-invariant, it is also $\overline{G}$-invariant. Thus, by Theorem \ref{closed minkowski} and Theorem \ref{closed Aleksandrov}, there exists a $\overline{G}$-invariant convex body $K$ such that
					\[
				\mu = \widetilde{C}_{p,q}(K, Q, \cdot) \quad \text{when $p \neq q$},
				\]
				and
				\[
				\mu = \lambda \widetilde{C}_{p,q}(K, Q, \cdot) \quad \text{for some $\lambda>0$ when $p=q$}.
				\]
				From Lemma \ref{ap2}, $K$ is also $G$-invariant, then the result is proved.
			\end{proof}
			
	When $n=1$, the $L_p$ dual Minkowski problem simplifies considerably. For $\mathrm{O}(1)$, since $\R$ has no nontrivial subspaces, its subgroups $\{\pm 1\}$ and $\{1\}$ are both irreducible. 
	\begin{itemize}
	\item	When $G=\{\pm 1\}$, let $Q=[-b, b]$ for $b>0$.
	For a  finite Borel measure $\mu$ on $S^{0}$, $G$-invariant implies $\mu(1)=\mu(-1)$. Let $K = [-a, a]$, then the Minkowski problem reduces to solving for  $a$. The equation $\mu = \widetilde{C}_{p,q}(K, Q, \cdot)$ is equivalent to
	$$\mu (1)=h_{[-a, a]}^{-p}(1)\rho_{[-a, a]}^{q}(1)\rho_{[-b, b]}^{1-q}(1)=a^{q-p}b^{1-q}.$$
	If $p\neq q$, then $a=(\frac{\mu(1)}{b^{1-q}})^{\frac{1}{q-p}}$, and $\mu = \widetilde{C}_{p,q}([-a,a], [-b, b], \cdot)$;
	If $p=q$, then $\mu = \lambda \widetilde{C}_{p,q}([-1,1], [-b, b], \cdot)$ for $\lambda=\frac{\mu(1)}{b^{1-q}}$.
	
	\item When $G=\{ 1\}$, let $Q=[-a, b]$ for $a, b>0$.
	For a  finite Borel measure $\mu$ on $S^{0}$, let $K = [-c, d]$; then the Minkowski problem reduces to solving for  $c$ and $d$. The equation $\mu = \widetilde{C}_{p,q}(K, Q, \cdot)$ is equivalent to
	$$\mu (1)=d^{q-p}b^{1-q}, \quad \mu (-1)=c^{q-p}a^{1-q}. $$
	If $p\neq q$, then $d=(\frac{\mu(1)}{b^{1-q}})^{\frac{1}{q-p}}$ and $c=(\frac{\mu(-1)}{a^{1-q}})^{\frac{1}{q-p}}$, and $\mu = \widetilde{C}_{p,q}([-c,d], [-a, b], \cdot)$.
	If $p=q$, the equation $\mu = \lambda \widetilde{C}_{p,q}(K, Q, \cdot)$
	 may have no solution.
	\end{itemize}
	
We now present a simple application of the existence theorem for solutions to the $\Lp$ dual Minkowski problem.

	\begin{example}
	Let $G$ be a subgroup of $\On$	such that for every $v \in \Sn$, the orbit $Gv$ is dense in $\Sn$. 
	Such a group $G$ is clearly irreducible. By Proposition \ref{open G ball}, every $G$-invariant convex body is a ball.  Let $Q = r_1 B$ be a $G$-invariant star body and $\mu$ a $G$-invariant finite Borel measure on $\Sn$. We seek a solution of the form $K = rB$. The $L_{p}$ dual Minkowski problem reduces to solving for  $r$. The equation $\mu = \widetilde{C}_{p,q}(K, Q, \cdot)$ is equivalent to
	\begin{equation}\label{application}
	\mu(\omega)=\frac{1}{n}\mathcal{H}^{n-1}(\omega)r_{1}^{n-q}r^{q-p}
\end{equation}
for any Borel set $\omega\subset \Sn$.	By Theorem \ref{general exist}, for $p\neq q$,  there exists a solution $r > 0$ satisfying equation \eqref{application}.   Moreover, from \eqref{application}, we conclude  that if $\mu$ is a finite $G$-invariant Borel measure where $G$ has dense orbits on $\Sn$, then $\mu$ must be a constant multiple of the spherical Hausdorff measure, i.e., $\mu = c \mathcal{H}^{n-1}$ for some $c > 0$.
	\end{example}

\end{document}